\documentclass[11pt,letterpaper, reqno]{amsart}

\newcommand{\la}{\langle}
\newcommand{\ra}{\rangle}

\def\R{{\mathbb R}}

\def\eqnn{\begin{eqnarray*}}
\def\eeqnn{\end{eqnarray*}}
\def\eqn{\begin{eqnarray}}
\def\eeqn{\end{eqnarray}}

\usepackage{amssymb}
\usepackage{amsthm}
\usepackage{amsxtra}
\usepackage{graphicx}
\usepackage{mathabx}

\newtheorem{theorem}{Theorem}

\newtheorem{proposition}[theorem]{Proposition}
\newtheorem{lemma}[theorem]{Lemma}
\newtheorem{corollary}[theorem]{Corollary}

\theoremstyle{remark}
\newtheorem{remark}[theorem]{Remark}

\numberwithin{equation}{section}

\numberwithin{theorem}{section}

\numberwithin{table}{section}

\numberwithin{figure}{section}

\ifx\pdfoutput\undefined
  \DeclareGraphicsExtensions{.pstex, .eps}
\else
  \ifx\pdfoutput\relax
    \DeclareGraphicsExtensions{.pstex, .eps}
  \else
    \ifnum\pdfoutput>0
      \DeclareGraphicsExtensions{.pdf}
    \else
      \DeclareGraphicsExtensions{.pstex, .eps}
    \fi
  \fi
\fi

\title[On the scattering problem for infinitely many fermions]{On the scattering problem for infinitely many fermions in dimensions $d\geq3$ at positive temperature}

\date{\today}
\author{Thomas Chen}
\address{T. Chen,  
Department of Mathematics, University of Texas at Austin.}
\email{tc@math.utexas.edu}
\author{Younghun Hong}
\address{Y. Hong,  
Department of Mathematics, Yonsei University, Seoul, 120-749, Republic of Korea.}
\email{younghun.hong@yonsei.ac.kr}
\author{Nata\v{s}a Pavlovi\'c}
\address{N. Pavlovi\'c,  
Department of Mathematics, University of Texas at Austin.}
\email{natasa@math.utexas.edu}

\begin{document}

\maketitle

\begin{abstract}
In this paper, we study the dynamics of a system of infinitely many fermions in dimensions $d\geq3$ near thermal equilibrium and prove scattering in the case of small perturbation around equilibrium in a certain generalized Sobolev space of density operators.  This work is a continuation of our previous paper \cite{CHP1}, and extends the important recent result of M. Lewin and J. Sabin in \cite{LS2} of a similar type for dimension $d=2$. In the work at hand, we establish new, improved Strichartz estimates that allow us to control the case $d\geq3$.
\end{abstract}

\section{Introduction}
\label{sec-intro}

In this paper, we study the dynamics of a system of infinitely many fermions in dimensions $d\geq3$ near thermal equilibrium. In particular, we prove scattering 
in the case when the perturbation around equilibrium is small in a certain generalized Sobolev space of density operators. This work is a continuation of our previous paper \cite{CHP1}, and extends some important recent result of M. Lewin and J. Sabin in \cite{LS2} of a similar type for two dimensions $(d=2)$. In the work at hand, we are employing new, improved Strichartz estimates that allow us to access higher dimensions.

To set up the problem, we start with a finite system of $N$ fermions interacting via a pair potential $w$  in mean-field description.
The dynamics is described by $N$ coupled Hartree equations 
\begin{equation}\label{NLHsysN}
\left\{\begin{aligned}
	i\partial_t u_1&=(-\Delta+w * \rho)u_1 \quad,& 
	\quad\quad  u_1(t=0)&=u_{1,0}\\
               &\cdots &   & \cdots\\
	i\partial_t u_N&=(-\Delta+w * \rho)u_N \quad,& \quad\quad  u_N(t=0)&=u_{N,0}
\end{aligned}
\right.
\end{equation}
where $\rho$ is the total density of particles
\begin{equation}\label{NLHsysN-rho}
	\rho(t,x)=\sum_{j=1}^N|u_j(t,x)|^2 \,.
\end{equation}
In order to be in agreement with the Pauli principle,
we require that the initial data $\{u_{j,0}\}_{j=1}^N$ 
is an orthonormal family. Given that the Cauchy problem is well-posed in a suitable solution space,
the solution  $\{u_{j,t}\}_{j=1}^N$ continues to be an orthonormal family for $t>0$.

We introduce the one-particle 
density matrix corresponding to \eqref{NLHsysN},
\begin{equation} \label{intro-gamma} 
\gamma_N(t) = \sum_{j=1}^N |u_{j}(t) \rangle\langle u_{j}(t)|.
\end{equation} 
It corresponds to the rank-$N$ orthogonal projection onto the 
span of the orthonormal family $\{u_{j}(t)\}_{j=1}^N$.
The system \eqref{NLHsysN} is then equivalent to a single operator-valued equation
\begin{equation}
    i\partial_t\gamma_N = [-\Delta+w * \rho_{\gamma_N}, \gamma_N]  \label{NLHsys1-N} 
\end{equation}
with initial data
\begin{equation}
    \gamma_N(t=0) =  \sum_{j=1}^N |u_{j,0}\rangle\langle u_{j,0}|, \label{NLHsys1-id-N}
\end{equation}
where the density function is given by
\begin{equation} \label{NLHsys1-rho-N}
    \rho_{\gamma_N}(t,x) = \gamma_N(t,x,x) \,.
\end{equation} 
Orthonormality of the family $\{u_{j}\}_{j=1}^N$  implies that $0\leq\gamma\leq1$.

The expected particle number $\int\rho_N dx$ diverges as $N\rightarrow\infty$
for the system \eqref{NLHsysN} - \eqref{NLHsysN-rho}, respectively \eqref{NLHsys1-N} - \eqref{NLHsys1-rho-N}.
Therefore,
the one-particle density matrix $\gamma=\sum_{j=1}^\infty|u_j\rangle\langle u_j|$
is {\em not} of trace class; on the other hand, it has a bounded operator norm $L^2\rightarrow L^2$. 

For a dilute gas with a finite density 
(for instance, with $\rho(t,x)=\frac1N\sum_{j=1}^N|u_j(t,x)|^2$ as $N\rightarrow\infty$,
or $\rho(t,x)=\sum_{j=1}^\infty \lambda_j|u_j(t,x)|^2$ with $\lambda_j>0$ and $\sum\lambda_j =1$),
the system \eqref{NLHsysN} 
has been extensively analyzed in the literature, see for instance
\cite{ACV1,BdPF-1, BdPF-2, BM1,Ch-1976, Z-1992}. 
In this setting, $\gamma=\lim_{N\rightarrow\infty}\gamma_N$ is trace class.
See also for instance
\cite{BGGM-2003, BEGMY-2002, EESY-2004, FK-2011, BNS,NS} 
and the references therein for its derivation
from a quantum system of interacting fermions;
we remark that the fermionic exchange term is negligible in this limit.

The Cauchy problem, obtained from \eqref{NLHsys1-rho-N} as $N\rightarrow\infty$ but with $\rho_\gamma\not\in L^1$, is much more difficult than in the earlier works noted above. The main problem is to understand in which framework the Cauchy problem
\begin{equation}\label{NLHsys1} 
i\partial_t\gamma = [-\Delta + w * \rho_{\gamma}, \gamma] 
\end{equation}
with initial data
\begin{equation}
    \gamma(0) =  \gamma_0 \, , \label{NLHsys1-id} 
\end{equation}
and density
\begin{equation}
	\rho_{\gamma}(t,x) = \gamma(t,x,x),
\end{equation}
can be meaningfully posed{\footnote{Again, it is required that $0 \leq \gamma_0 \leq 1$, to be in agreement with the Pauli principle; hence,
$\gamma$ has a bounded operator norm.}}. Lewin and Sabin were the first authors who introduced a framework for this problem \cite{LS1,LS2}, which can be described as follows. First, we observe that given a non-negative function $f:\mathbb{R}\to\mathbb{R}_{\geq0}$, the operator $\gamma_f=f(-\Delta)$ is a stationary solution to \eqref{NLHsys1} having infinite particle number, i.e., $\rho_{\gamma_f}\notin L^1$, since the density function $\rho_{\gamma_f}$ is a constant function. Examples of  $\gamma_f$ include the Fermi sea of the non-interacting system. 
For inverse temperature $\beta>0$ and chemical 
potential $\mu>0$, the Fermi sea $\gamma_f$ is given by the Fermi-Dirac distribution
\begin{equation}\label{Fermi-Dirac}
    \gamma_f(x,y) = \int_{\R^d} \frac{e^{ip(x-y)}}{e^{\beta(p^2-\mu)}+1}  dp = 
    \Big(\frac{1}{e^{\beta(-\Delta-\mu)}+1}\Big)(x,y) \,.
\end{equation}
while in the zero temperature limit,
\begin{equation}
    \gamma_f = \Pi_\mu^-=\mathbf{1}_{(-\Delta\leq\mu)} \,.
\end{equation}
Then, the main idea is to consider a perturbation 
\begin{equation}
	Q:=\gamma-\gamma_f
\end{equation}
from the reference state $\gamma_f$, which evolves according to the following Cauchy problem: 
\begin{equation}\label{Q equation}
\left\{\begin{aligned}
i\partial_tQ & =[-\Delta+w*\rho_Q,Q+\gamma_f],\\
Q(0)& =Q_0.
\end{aligned}
\right.
\end{equation} 

In \cite{LS1}, Lewin and Sabin proved that the Cauchy problem \eqref{Q equation} 
for  $Q$ is globally well-posed for $d\geq 2$ in a suitable subspace
of the space of compact operators, provided that the pair interaction $w$
is sufficiently regular. An important tool used in \cite{LS1} was 
a Strichartz estimate for density functions originally established in \cite{FLLS2}, which is extended to the optimal range \cite{FS}. The case of a more singular interaction potential, with $w=\delta$ given by the Dirac delta, was 
analyzed by authors of the paper at hand; in \cite{CHP1}, we proved global well-posedness of the 
perturbative system \eqref{Q equation}, at zero temperature $\gamma_f=\Pi_\mu^-$, by employing new Strichartz estimates for {\em regular} density functions and those for operator kernels, 
which were established in the same paper \cite{CHP1}. 

In the case of a sufficiently regular potential $w$, Lewin-Sabin in \cite{LS2} 
proved scattering for $Q$ in $d=2$ via Strichartz estimates from \cite{FLLS2}.  
The case of higher dimensions was left open, and the purpose of the paper at hand is to address it.

Before we state the main result of this paper in Theorem \ref{thm-main-1},
we present a brief review of the notation. For $p \geq 1$, the  Schatten class $\mathfrak{S}^{p}$ is defined via 
$$\|A\|_{\mathfrak{S}^{p}}=({\rm Tr}(|A|^p))^{1/p},$$ 
while for $\alpha \geq 0$
a Hilbert-Schmidt Sobolev space $\mathcal{H}^\alpha$ is equipped with the norm\footnote{For details, see \eqref{sec3-defHS}.}  
$$\|Q\|_{\mathcal{H}^\alpha}=\|\la\nabla\ra^\alpha Q\la\nabla\ra^\alpha\|_{\mathfrak{S}^2}.$$
Also, we use the standard notation %$\check{g}$ 
$$ \check{g}(x) = \frac{1}{(2\pi)^d} \int e^{ix\cdot \xi} g(\xi) \; d\xi $$
to denote the inverse Fourier transform of a function $g$.

\begin{theorem}
\label{thm-main-1}
Let $d\geq 3$, $\alpha>\frac{d-2}{2}$ and $\alpha_0$ be given by 
\begin{equation}\label{non-relativistic beta-thst}
\left\{\begin{aligned}
&\alpha_0=2\alpha-\tfrac{d-1}{2}&&\textup{if }\alpha<\tfrac{d-1}{2},\\
&\alpha_0<\alpha&&\textup{if }\alpha=\tfrac{d-1}{2},\\
&\alpha_0=\alpha&&\textup{if }\alpha>\tfrac{d-1}{2}
\end{aligned}
\right.
\end{equation}
and $\beta>\frac{d+2}{2}$.

We assume that\\
$(i)$ (assumptions on $f$)
$f$ is real-valued, $\langle\cdot\rangle^\beta f\in L_{r\geq 0}^\infty$, $f'(r)<0$ for $r>0$,
\begin{equation}
\int_0^\infty (r^{d/2-1}|f(r)|+|f'(r)|)dr<\infty\quad \textup{and}\quad \int_{\mathbb{R}^d} \frac{\check{g}(x)}{|x|^{d-2}}dx<\infty,
\end{equation}
where $g(\xi)=f(|\xi|^2)$.\\
$(ii)$ (assumption on $w$) The interaction potential $w=w_1*w_2\in L^1$ is even,
\begin{equation}\label{assumption on w_1}
\|\hat{w}_1\|_{L^{\frac{2d}{d-2}}},\ \|\hat{w}_2\|_{L^{\frac{2d}{d-2}}},\ \|\langle\cdot\rangle^{\alpha_0+\frac{1}{2}}\hat{w}_1\|_{L^\infty},\|\hat{w}_2\|_{L^\infty}, \ \||\cdot|^{-1/2}\langle\cdot\rangle^{-\alpha_0}\hat{w}_2\|_{L^\infty}<\infty,
\end{equation}
and
\begin{equation}\label{assumption on w_2}
\|\hat{w}_-\|_{L^\infty}<2|\mathbb{S}^{d-1}|\Big(\int_{\mathbb{R}^d}\frac{|\check{g}(x)|}{|x|^{d-2}}dx\Big)^{-1}\quad\textup{and}\quad
\hat{w}_+(0)<\frac{2}{\epsilon_g}|\mathbb{S}^{d-1}|,
\end{equation}
where $A_\pm=\max\{\pm A,0\}$ and $\epsilon_g$ is given by \eqref{e_g}. 

Then, there exists small $\epsilon>0$ such that if $\|Q_0\|_{\mathcal{H}^\alpha}\leq\epsilon$, there exists a unique global solution  $Q(t)\in C_t(\mathbb{R};\mathfrak{S}^{2d})$ to the equation \eqref{Q equation} with initial data $Q_0$. Moreover, the associated density function $\rho_Q$ obeys the global space-time bound, 
\begin{equation}\label{eq-dens-globbd-1}
\|w*\rho_Q\|_{L_t^2(\mathbb{R}; L_x^d)}<\infty,
\end{equation}
and $Q(t)$ scatters in $\mathfrak{S}^{2d}$ as $t\to\pm\infty$; in other words, there exist $Q_{\pm}\in\mathfrak{S}^{2d}$ such that $e^{-it\Delta}Q(t)e^{it\Delta}\rightarrow Q_{\pm}$ converges strongly in $\mathfrak{S}^{2d}$ as $t\to\pm\infty$.
\end{theorem}

\begin{remark}
$(i)$ In Theorem \ref{thm-main-1}, various conditions are imposed on the reference state $\gamma_f$, the interaction potential $w$, and the initial data $Q_0$. Our main goal is to prove scattering in high dimensions. We do not pursue any optimality on the hypotheses. Some physically important examples, such as the Fermi-Dirac distribution \eqref{Fermi-Dirac}, satisfy these assumptions. The assumptions on $f$ and the assumptions in \eqref{assumption on w_2} are used for the linear response theory (see Proposition \ref{prop: invertibility}). The assumptions in \eqref{assumption on w_1} are used for the proof of the global space-time bound \eqref{eq-dens-globbd-1} (see Section \ref{proof of the main theorem}).\\
$(ii)$ The method in our paper might be applied to the two-dimensional case with different conditions on the interaction potential $w$ and initial data $Q_0$ from Lewin and Sabin \cite{LS2}. However, we omit the case $d=2$, as it was already proved in \cite{LS2}; moreover, some exponents would have to be modified in the proof. For instance, we are using the endpoint Strichartz estimate for convenience, but the endpoint estimate is known to be false in $\mathbb{R}^2$ \cite{Tao}.\\
$(iii)$ As a crucial new ingredient that allow us to 
extend the work of 
Lewin-Sabin \cite{LS2} to dimensions higher than 2, 
we establish new 
Strichartz estimates for density functions and density matrices  in Section \ref{sec-Strichartz-density-1}. Compared to the Strichartz estimates derived in \cite{FLLS2}, and used in \cite{LS2}, our Strichartz estimates exhibit an improved summability gain by imposing more regularity on the initial data.
\end{remark}

 \subsection*{Acknowledgements}
The work of T.C. was supported by NSF CAREER grant DMS-1151414. 
The work of Y.H. was supported by NRF grant 2015R1A5A1009350.
The work of N.P. was supported by NSF grant DMS-1516228.

\section{Outline of the Proof of Theorem \ref{thm-main-1}}

In this part of our analysis, we explain the strategy to prove the main result of this article, Theorem \ref{thm-main-1}. First, in Section \ref{strategy 1}, we show that if the density function $\rho_Q$ of the solution to \eqref{Q equation} satisfies the global space-time bound (see \eqref{global bound for density}), then the solution $Q(t)$ scatters. Next, in \S \ref{strategy 2}, we set up a suitable contraction map $\Gamma$ (see \eqref{Gamma}) to construct a solution obeying the desired global space-time bound.

\subsection{A global space-time bound for a density function implies scattering}\label{strategy 1}
We follow the strategy in Lewin and Sabin \cite{LS2}. For simplicity, we present the argument only for the forward-in-time direction, as it can be easily modified to prove scattering backward in time.

Given a time-dependent potential $V=V(t,x)$, we denote by $\mathcal{U}_V(t)$ the linear propagator for the linear Schr\"odinger equation
\begin{equation}\label{linear Schrodinger}
i\partial_t u+\Delta u- Vu=0,
\end{equation}
i.e., $\mathcal{U}_V(t)\phi$ is the solution to \eqref{linear Schrodinger} with initial data $\phi$. We define the \textit{``finite-time" wave operator} $\mathcal{W}_V(t)$ by
\begin{equation}\label{eq-W-def-1}
\mathcal{W}_V(t):=e^{-it\Delta}\mathcal{U}_V(t).
\end{equation}
Iterating the Duhamel formula
\begin{equation}
\mathcal{U}_V(t)=e^{it\Delta}-i\int_0^t e^{i(t-t_1)\Delta} V(t_1)\mathcal{U}_V(t_1)dt_1
\end{equation}
infinitely many times, the wave operator can be written as an infinite sum, 
\begin{equation}\label{wave operator sum}
\mathcal{W}_V(t):=\sum_{n=0}^\infty \mathcal{W}_V^{(n)}(t),
\end{equation}
where $\mathcal{W}_V^{(0)}(t):=\textup{Id}$, and for $n\geq 1$,
\begin{equation}\label{W_V^n}
\begin{aligned}
\mathcal{W}_V^{(n)}(t):&=(-i)^n\int_0^tdt_n\int_0^{t_n}dt_{n-1}\cdots \int_0^{t_2}dt_1e^{-it_n\Delta}V(t_n) e^{it_n\Delta}\\
&\quad\quad\quad\quad\quad\quad\quad\quad \cdot e^{-it_{n-1}\Delta} V(t_{n-1})e^{it_{n-1}\Delta}\cdots e^{-it_1\Delta}V(t_1)e^{it_1\Delta}
\\
&=(-i)\int_0^tdt_n e^{-it_n\Delta}V(t_n) e^{it_n\Delta} \mathcal{W}_V^{(n-1)}(t_n).
\end{aligned}
\end{equation}
By the definition of the finite-time wave operator, the equation \eqref{Q equation} is equivalent to
\begin{equation}\label{wave operator formulation}
Q(t)=e^{it\Delta}\mathcal{W}_{w*\rho_{Q}}(t)(\gamma_f+Q_0)\mathcal{W}_{w*\rho_{Q}}(t)^* e^{-it\Delta}-\gamma_f,
\end{equation}
because $Q(t)=\gamma(t)-\gamma_f$ and
\begin{equation}
\gamma(t)=\mathcal{U}_V(t)\gamma_0 \mathcal{U}_V(t)^*.
\end{equation}
Inserting the sum \eqref{wave operator sum} into the equation \eqref{wave operator formulation}, it becomes
\begin{equation}\label{series expansion for Q}
\begin{aligned}
Q(t)&=e^{it\Delta}\Big(\sum_{m=0}^\infty \mathcal{W}_{w*\rho_Q}^{(m)}(t)\Big)\gamma_f\Big(\sum_{n=0}^\infty \mathcal{W}_{w*\rho_Q}^{(n)}(t)\Big)^*e^{-it\Delta}-\gamma_f\\
&\quad+e^{it\Delta}\Big(\sum_{m=0}^\infty \mathcal{W}_{w*\rho_Q}^{(m)}(t)\Big)Q_0\Big(\sum_{n=0}^\infty \mathcal{W}_{w*\rho_Q}^{(n)}(t)\Big)^*e^{-it\Delta}\\
&=e^{it\Delta}Q_0e^{-it\Delta}+\sum_{(m,n)\neq (0,0)} e^{it\Delta}\mathcal{W}_{w*\rho_Q}^{(m)}(t)\gamma_f\mathcal{W}_{w*\rho_Q}^{(n)}(t)^*e^{-it\Delta}\\
&\quad+\sum_{(m,n)\neq (0,0)}e^{it\Delta}\mathcal{W}_{w*\rho_Q}^{(m)}(t) Q_0 \mathcal{W}_{w*\rho_Q}^{(n)}(t)e^{-it\Delta}.
\end{aligned}
\end{equation}
In \cite{FLLS2}, Frank, Lewin, Lieb and Seiringer prove that if $d\geq 2$, then 
\begin{equation}\label{boundedness of wave operator}
\|\mathcal{W}_V^{(n)}(t_0)\|_{\mathfrak{S}^{2d}}\leq \frac{1}{(n!)^{\frac{1}{2}-\epsilon}}\big(C\|V\|_{L_{[0,+\infty)}^2L_x^d}\big)^n,\quad\forall n\geq 1
\end{equation}
for any small $\epsilon>0$ (see Theorem 2 for $n=1$ and Theorem 3 for $n\geq 2$ in \cite{FLLS2}). Therefore, if the density function obeys the space-time norm bound
\begin{equation}\label{global bound for density}
\|w*\rho_Q\|_{L_t^2([0,+\infty); L_x^d)}<\infty,
\end{equation}
by \eqref{boundedness of wave operator} with $V=w*\rho_Q$, the series is absolutely convergent in $C_t([0,+\infty);\mathfrak{S}^{2d})$, so it is well-defined.

Using this series expansion, we prove that the global space-time bound \eqref{global bound for density} implies scattering.
\begin{lemma}[A global space-time bound for a density function implies scattering]\label{key lemma}
Let $d\geq 3$. Suppose that $Q(t)\in C_t([0,+\infty);\mathfrak{S}^{2d})$ is a solution to the equation \eqref{wave operator formulation} and its density satisfies \eqref{global bound for density}. Then, $Q(t)$ scatters in $\mathfrak{S}^{2d}$ as $t\to+\infty$.
\end{lemma}

\begin{proof}
As in the proof of the absolute convergence of the series, applying the inequality \eqref{boundedness of wave operator} to the series expansion of the difference between $e^{-it_1\Delta}Q(t)e^{it_1\Delta}$ and $e^{-it_2\Delta}Q(t)e^{it_2\Delta}$, one can show that $e^{-it_1\Delta}Q(t)e^{it_1\Delta}-e^{-it_2\Delta}Q(t)e^{it_2\Delta}\to 0$ in $\mathfrak{S}^{2d}$ as $t_1, t_2\to+\infty$. Therefore, $e^{-it\Delta}Q(t)e^{it\Delta}$ has a strong limit $Q_+$ in $\mathfrak{S}^{2d}$ as $t\to+\infty$. That is, $Q(t)$ scatters in $\mathfrak{S}^{2d}$ as $t\to+\infty$.
\end{proof}

\subsection{Set-up for the contraction mapping argument}\label{strategy 2}
By Lemma \ref{key lemma}, the goal is now to prove that the equation \eqref{wave operator formulation} has a unique solution $Q(t)$ in a suitable space obeying the space-time bound \eqref{global bound for density}. 
To this end, as in Lewin-Sabin \cite{LS2}, we write the equation \eqref{wave operator formulation} as an equation for density functions,
\begin{equation}\label{equation for density before expansion}
\rho_{Q(t)}=\rho\Big[e^{it\Delta}\mathcal{W}_{w*\rho_{Q}}(t)(\gamma_f+Q_0)\mathcal{W}_{w*\rho_{Q}}(t)^* e^{-it\Delta}\Big]-\rho_{\gamma_f}.
\end{equation}
One of the advantages of this wave operator formulation in density is that the unknown is given only by the density function, and there is no unknown operator. 

We further simplify the equation by splitting the interaction potential $w$ into $w=w_1*w_2$, and subsequently convolving the density function $\rho_Q$ with $w_2$, 
\begin{equation}
w_2*\rho_{Q(t)}=w_2*\rho\Big[e^{it\Delta}\mathcal{W}_{w_1*(w_2*\rho_{Q})}(t)(\gamma_f+Q_0)\mathcal{W}_{w_1*(w_2*\rho_{Q})}(t)^* e^{-it\Delta}\Big]-w_2*\rho_{\gamma_f}.
\end{equation}
Now we consider the equation for $w_2*\rho_Q$. The motivation for this formulation is that the solution $w_2*\rho_Q$ is expected to be contained in a larger function space (or bounded in a weaker norm) than the one for $\rho_Q$, provided that $w_2$ is sufficiently nice; our constructions will exploit this fact. 

Next, inserting the sum \eqref{wave operator sum} for the finite time wave operators acting on $\gamma_f$, we write
\begin{equation}\label{equation for density}
\begin{aligned}
w_2*\rho_{Q(t)}&=w_2*\rho\Big[e^{it\Delta}\Big(\sum_{m=0}^\infty \mathcal{W}_{w_1*(w_2*\rho_{Q})}^{(m)}(t)\Big)\gamma_f\Big(\sum_{n=0}^\infty \mathcal{W}_{w_1*(w_2*\rho_{Q})}^{(n)}(t)\Big)^* e^{-it\Delta}\Big]\\
&\quad+w_2*\rho\Big[e^{it\Delta}\mathcal{W}_{w_1*(w_2*\rho_{Q})}(t)Q_0\mathcal{W}_{w_1*(w_2*\rho_{Q})}(t)^* e^{-it\Delta}\Big]-w_2*\rho_{\gamma_f}\\
&=w_2*\rho\Big[e^{it\Delta}\Big(\mathcal{W}_{w_1*(w_2*\rho_{Q})}^{(1)}(t)\gamma_f+\gamma_f\mathcal{W}_{w_1*(w_2*\rho_{Q})}^{(1)}(t)^*\Big) e^{-it\Delta}\Big]\\
&\quad+\sum_{m,n=1}^\infty w_2*\rho\Big[e^{it\Delta}\mathcal{W}_{w_1*(w_2*\rho_{Q})}^{(m)}(t)\gamma_f\mathcal{W}_{w_1*(w_2*\rho_{Q})}^{(n)}(t)^* e^{-it\Delta}\Big]\\
&\quad+w_2*\rho\Big[e^{it\Delta}\mathcal{W}_{w_1*(w_2*\rho_{Q})}(t)Q_0\mathcal{W}_{w_1*(w_2*\rho_{Q})}(t)^* e^{-it\Delta}\Big].
\end{aligned}
\end{equation}
Then, introducing the operators, 
\begin{align}
\mathcal{L}(\phi)(t):&=-w_2*\rho\Big[e^{it\Delta}\Big(\mathcal{W}_{w_1*\phi}^{(1)}(t)\gamma_f+\gamma_f\mathcal{W}_{w_1*\phi}^{(1)}(t)^*\Big) e^{-it\Delta}\Big],\label{L operator}\\
\mathcal{A}_{m,n}(\phi)(t):&=w_2*\rho\Big[e^{it\Delta}\mathcal{W}_{w_1*\phi}^{(m)}(t)\gamma_f\mathcal{W}_{w_1*\phi}^{(n)}(t)^* e^{-it\Delta}\Big],\label{A operator}\\
\mathcal{B}(\phi)(t):&=w_2*\rho\Big[e^{it\Delta}\mathcal{W}_{w_1*\phi}(t)Q_0\mathcal{W}_{w_1*\phi}(t)^* e^{-it\Delta}\Big],\label{B operator}
\end{align}
we write
\begin{equation}\label{our formulation}
w_2*\rho_{Q}=-\mathcal{L}(w_2*\rho_Q)+\Big\{\sum_{m,n=1}^\infty\mathcal{A}_{m,n}(w_2*\rho_Q)+\mathcal{B}(w_2*\rho_Q)\Big\}.
\end{equation}
We note that compared to the formulation in \cite{LS2}, the equation \eqref{our formulation} is slightly simpler in that $\mathcal{B}(w_2*\rho_Q)$ is not expanded as an infinite sum. However, due to the linear nature of the operator $\mathcal{L}$, which is not perturbative even for small functions, the series expansion $\sum_{m,n=1}^\infty\mathcal{A}_{m,n}(w_2*\rho_Q)$ does not seem to be avoidable. 

Later in Section \ref{invertibility}, it will be shown that $(1+\mathcal{L})$ is invertible on $L_{t\geq 0}^2L_x^2$. As a result, the equation can be reformulated as
\begin{equation}\label{wave equation formulation}
w_2*\rho_{Q}=(1+\mathcal{L})^{-1}\Big\{\sum_{m,n=1}^\infty\mathcal{A}_{m,n}(w_2*\rho_Q)+\mathcal{B}(w_2*\rho_Q)\Big\}.
\end{equation}
Our goal is now to show that the map $\Gamma$, defined by
\begin{equation}\label{Gamma}
\Gamma(\phi)=(1+\mathcal{L})^{-1}\Big\{\sum_{m,n=1}^\infty\mathcal{A}_{m,n}(\phi)+\mathcal{B}(\phi)\Big\},
\end{equation}
is contractive in a suitable function space, and its solution satisfies the space-time bound
\begin{equation}
\|\phi\|_{L_{t\geq 0}^2L_x^2}<\infty.
\end{equation}
Then, the main theorem follows (see Section \ref{proof of the main theorem}).

\section{Strichartz estimates for density functions}
\label{sec-Strichartz-density-1}

In this section we present the Strichartz estimates that will be used in our analysis.
First, we give an overview of the notation. 

\subsection{Notation} 
As already mentioned in Section \ref{sec-intro}, 
we denote by $\mathfrak{S}^p$ the Schatten spaces, equipped with the norms
\begin{equation}
    \|Q\|_{\mathfrak{S}^p} := \Big({\rm Tr}|Q|^p \Big)^{1/p},
\end{equation}
for $p\geq 1$.

For $\alpha\geq 0$, we define the \textit{Hilbert-Schmidt Sobolev space} $\mathcal{H}^\alpha$ as the collection of Hilbert-Schmidt operators (which are not necessarily self-adjoint) with a finite norm
\begin{equation} \label{sec3-defHS} 
\|\gamma_0\|_{\mathcal{H}^\alpha}:=\|\la\nabla\ra^\alpha\gamma_0\la\nabla\ra^\alpha\|_{\mathfrak{S}^2}=\|\la\nabla_x\ra^\alpha\la\nabla_{x'}\ra^\alpha \gamma_0(x,x')\|_{L_x^2L_{x'}^2}.
\end{equation}
Here, $\gamma_0(x,x')$ is the integral kernel of $\gamma_0$, i.e.,  
\begin{equation}
(\gamma_0 g)(x)=\int_{\mathbb{R}^d}\gamma_0(x,x')g(x')dx'.
\end{equation}

In order to review Strichartz estimates for operator kernels in Subsection \ref{subs3-Str}, 
we need to  recall notation from \cite{CHP1} related to Strichartz norms. 
An exponent pair $(q,r)$ is \textit{(Strichartz) admissible} if $2\leq q,r\leq\infty$, $(q,r,d)\neq (2,\infty,2)$ and 
\begin{equation}
\frac{2}{q}+\frac{d}{r}=\frac{d}{2}.
\end{equation}
Assume that $\gamma(t)$ is a time-dependent operator  on an interval $I\subset\mathbb{R}$. 
Then, its Strichartz norm is defined by
\begin{equation}
\begin{aligned}\label{eq-StrichNorm-def-1}
    \|\gamma(t)\|_{\mathcal{S}^\alpha(I)}:=\sup_{(q,r)\textup{: admissible}}&
    \Big\{\|\la\nabla_x\ra^\alpha\la\nabla_{x'}\ra^\alpha\gamma(t,x,x')\|_{L_{t}^q(I; L_x^rL_{x'}^2)}\\
    &+\|\la\nabla_x\ra^\alpha\la\nabla_{x'}\ra^\alpha\gamma(t,x,x')\|_{L_{t}^q (I;L_{x'}^rL_{x}^2)}\Big\}.
\end{aligned}
\end{equation}
It is clear that $\mathcal{S}^\alpha(I) \hookrightarrow L_t^\infty(I;\mathcal{H}^\alpha)$.

We identify the operator $e^{it\Delta}\gamma_0 e^{-it\Delta}$ with its integral kernel
\begin{equation}
    (e^{it\Delta}\gamma_0 e^{-it\Delta})(x,x')=(e^{it(\Delta_x-\Delta_{x'})}\gamma_0)(x,x').
\end{equation}

\subsection{Strichartz estimates for density functions}
In this section, we prove new Strichartz estimates for density functions, which extend Strichartz estimates proved in the authors' previous work  \cite{CHP1} by allowing 
asymmetric derivatives ($\alpha_1$ not necessarily equal to $\alpha_2$). Those are presented in Theorem \ref{Strichartz estimates for density functions in the non-relativistic case}, and as a main application, we obtain Corollary  \ref{Bound on IV}, which we use to control the operators $\mathcal{A}_{m,n}$.

\begin{theorem}[Strichartz estimates for density functions]\label{Strichartz estimates for density functions in the non-relativistic case}
Suppose that $\alpha_0, \alpha_1,\alpha_2\geq 0$. When $d=1$, we assume that $\alpha=\min\{\alpha_1,\alpha_2\}$. When $d\geq 2$, we assume that
\begin{equation}\label{non-relativistic alpha}
\alpha_1+\alpha_2>\tfrac{d-1}{2}
\end{equation}
and
\begin{equation}\label{non-relativistic beta}
\left\{\begin{aligned}
&\alpha_0=\alpha_1+\alpha_2-\tfrac{d-1}{2}&&\textup{if }\max\{\alpha_1,\alpha_2\}<\tfrac{d-1}{2},\\
&\alpha_0<\min\{\alpha_1,\alpha_2\}&&\textup{if }\max\{\alpha_1,\alpha_2\}=\tfrac{d-1}{2},\\
&\alpha_0=\min\{\alpha_1,\alpha_2\}&&\textup{if }\max\{\alpha_1,\alpha_2\}>\tfrac{d-1}{2}.
\end{aligned}
\right.
\end{equation}
Then,
\begin{equation}\label{non-relativistic Strichartz}
\big\||\nabla|^{1/2} \rho_{e^{it\Delta}\gamma_0 e^{-it\Delta}}\big\|_{L_{t\in\mathbb{R}}^2H_x^{\alpha_0}}\lesssim \|\la\nabla\ra^{\alpha_1}\gamma_0\la\nabla\ra^{\alpha_2}\|_{\mathfrak{S}^2}.
\end{equation}
\end{theorem}

\begin{corollary}\label{Bound on IV} Suppose that $\alpha_0$, $\alpha_1$ and $\alpha_2$ satisfy the assumptions in Theorem \ref{Strichartz estimates for density functions in the non-relativistic case}. Then, 
\begin{equation}\label{eq: Bound on IV}
\Big\|\la\nabla\ra^{-\alpha_1}\int_{\mathbb{R}}e^{-it\Delta}V(t)e^{it\Delta}dt \la\nabla\ra^{-\alpha_2}\Big\|_{\mathfrak{S}^2}\leq c \|V(t)\|_{L_t^2L_x^{\frac{2d}{d+1}}}.
\end{equation}
\end{corollary}

\begin{proof}[Proof of Corollary \ref{Bound on IV}, assuming Theorem \ref{Strichartz estimates for density functions in the non-relativistic case}]
For a compactly supported smooth function $V(t,x)$ and a finite rank smooth operator $\gamma_0$, we write
\begin{equation}
\begin{aligned}
&\textup{Tr}\Big(\la\nabla\ra^{-\alpha_1}\int_{\mathbb{R}}e^{-it\Delta}V(t)e^{it\Delta}dt \la\nabla\ra^{-\alpha_2}\Big)\gamma_0\\
&=\int_{\mathbb{R}}\textup{Tr}\Big(e^{it\Delta}\la\nabla\ra^{-\alpha_2}\gamma_0\la\nabla\ra^{-\alpha_1} e^{-it\Delta}V(t)\Big)dt\\
&=\int_{\mathbb{R}}\int_{\mathbb{R}^d} \rho_{e^{it\Delta}\la\nabla\ra^{-\alpha_2}\gamma_0\la\nabla\ra^{-\alpha_1} e^{-it\Delta}}(x) V(t,x) dxdt,
\end{aligned}
\end{equation}
where the first identity is from cyclicity of trace. Therefore, \eqref{eq: Bound on IV} is dual to 
\begin{equation}\label{asymmetric Strichartz estimates for density}
\|\rho_{e^{it\Delta}\gamma_0 e^{-it\Delta}}\|_{L_{t}^2L_x^{\frac{2d}{d-1}}}\leq c\|\la\nabla\ra^{\alpha_2}\gamma_0\la\nabla\ra^{\alpha_1}\|_{\mathfrak{S}^2},
\end{equation}
which follows from \eqref{non-relativistic Strichartz} and the Sobolev inequality.
\end{proof}

The main strategy to prove the Strichartz estimate for density functions is to reformulate it as an integral estimate through the space-time Fourier transformation. This approach, via bilinear estimates based on the space-time $L^2$-norm, has been introduced by Klainerman and Machedon \cite{KM1,KM2}, and subsequently developed by many authors.

\begin{lemma}[Reduction to an integral estimate]\label{reduction lemma}
Let $\tilde{\alpha}$ be any real number. Then if 
 the integral 
\begin{equation}\label{integral estimate}
I_{\tau,\xi}:=\int_{|\eta|\leq|\xi-\eta|}\frac{|\xi|^{2\tilde{\alpha}}\la\xi\ra^{2\alpha_0}}{\la\eta\ra^{2\max\{\alpha_1,\alpha_2\}}\la\xi-\eta\ra^{2\min\{\alpha_1,\alpha_2\}}}\delta(\tau+|\eta|^2-|\xi-\eta|^2) d\eta
\end{equation}
is bounded uniformly in $\tau$ and $\xi$, 
the Strichartz estimate
\begin{equation}\label{Strichartz estimate: reduction lemma}
\||\nabla|^{\tilde{\alpha}}\rho_{e^{it\Delta}\gamma_0 e^{-it\Delta}}\|_{L_{t\in\mathbb{R}}^2H_x^{\alpha_0}}\lesssim \|\la\nabla\ra^{\alpha_1}\gamma_0\la\nabla\ra^{\alpha_2}\|_{\mathfrak{S}^2}
\end{equation}
holds.
\end{lemma}

\begin{proof}
The Fourier transform of the density function of $\gamma$ is given by
\begin{equation}\label{Fourier transform of density}
\begin{aligned}
\widehat{\rho_\gamma}(\xi)&=\mathcal{F}_x\Big\{\frac{1}{(2\pi)^{2d}}\int_{\mathbb{R}^d}\int_{\mathbb{R}^d}\hat{\gamma}(\eta,\zeta) e^{ix\cdot(\eta+\zeta)}dx d\zeta\Big\}(\xi)\\
&=\frac{1}{(2\pi)^{2d}}\int_{\mathbb{R}^d}\int_{\mathbb{R}^d}\hat{\gamma}(\eta,\zeta) \mathcal{F}_x\Big\{e^{ix\cdot(\eta+\zeta)}\Big\}(\xi)d\eta d\zeta\\
&=\frac{1}{(2\pi)^{2d}}\int_{\mathbb{R}^d}\int_{\mathbb{R}^d}\hat{\gamma}(\eta,\zeta)\cdot(2\pi)^d\delta(\xi-\eta-\zeta)d\eta d\zeta\\
&=\frac{1}{(2\pi)^{d}}\int_{\mathbb{R}^d}\hat{\gamma}(\eta,\xi-\eta)d\eta.
\end{aligned}
\end{equation}
Hence, the space-time Fourier transform of the density function $\rho_{e^{it\Delta}\gamma e^{-it\Delta}}$ is
\begin{equation}\label{space-time Fourier transform of density}
\begin{aligned}
(\rho_{e^{it\Delta}\gamma_0 e^{-it\Delta}})^\sim(\tau,\xi)&=\frac{1}{(2\pi)^{d}}\int_{\mathbb{R}^d}\mathcal{F}_t\Big\{e^{-it(|\eta|^2-|\xi-\eta|^2)}\Big\}\hat{\gamma}_0(\eta,\xi-\eta)d\eta\\
&=\frac{1}{(2\pi)^{d-1}}\int_{\mathbb{R}^d}\delta(\tau+|\eta|^2-|\xi-\eta|^2)\hat{\gamma}_0(\eta,\xi-\eta)d\eta.
\end{aligned}
\end{equation}
Thus, by the Plancherel theorem and Cauchy-Schwarz, we get
\begin{equation}
\begin{aligned}
&\||\nabla|^{\tilde{\alpha}}\rho_{e^{it\Delta}\gamma_0 e^{-it\Delta}}\|_{L_{t\in\mathbb{R}}^2H_x^{\alpha_0}}^2\\
&=\frac{1}{(2\pi)^{2(d+1)}}\Big\||\xi|^{\tilde{\alpha}}\la\xi\ra^{\alpha_0}(\rho_{e^{it\Delta}\gamma_0 e^{-it\Delta}})^\sim(\tau,\xi)\Big\|_{L_{\tau\in\mathbb{R}}^2L_{\xi}^2}^2\\
&=\frac{1}{(2\pi)^{2(d+1)}}\int_{\mathbb{R}}\int_{\mathbb{R}^d} |\xi|^{2\tilde{\alpha}}\la\xi\ra^{2\alpha_0}\\
&\quad\quad\quad\quad\quad\quad\quad\cdot\Big|\frac{1}{(2\pi)^{d-1}}\int_{\mathbb{R}^d}\delta(\tau+|\eta|^2-|\xi-\eta|^2)\hat{\gamma}_0(\eta,\xi-\eta)d\eta\Big|^2d\xi d\tau\\
&\leq\frac{1}{(2\pi)^{4d}}\int_{\mathbb{R}}\int_{\mathbb{R}^d} |\xi|^{2\tilde{\alpha}}\la\xi\ra^{2\alpha_0} \Big\{\int_{\mathbb{R}^d}\frac{\delta(\tau+|\eta|^2-|\xi-\eta|^2)}{\la\eta\ra^{2\alpha_1}\la\xi-\eta\ra^{2\alpha_2}}d\eta\Big\}\\
&\quad\quad\quad\quad\quad\cdot \Big\{\int_{\mathbb{R}^d}\delta(\tau+|\eta|^2-|\xi-\eta|^2)|(\la\nabla\ra^{\alpha_1}\gamma_0\la\nabla\ra^{\alpha_2})^\wedge(\eta,\xi-\eta)|^2d\eta\Big\}d\xi d\tau\\
&\leq \sup_{\tau,\xi}\frac{1}{(2\pi)^{4d}}\Big\{\int_{\mathbb{R}^d}\frac{|\xi|^{2\tilde{\alpha}}\la\xi\ra^{2\alpha_0}\delta(\tau+|\eta|^2-|\xi-\eta|^2)}{\la\eta\ra^{2\alpha_1}\la\xi-\eta\ra^{2\alpha_2}}d\eta\Big\}\\
&\quad\quad\quad\cdot \int_{\mathbb{R}}\int_{\mathbb{R}^d}\int_{\mathbb{R}^d}\delta(\tau+|\eta|^2-|\xi-\eta|^2)|(\la\nabla\ra^{\alpha_1}\gamma_0\la\nabla\ra^{\alpha_2})^\wedge(\eta,\xi-\eta)|^2d\eta d\xi d\tau.
\end{aligned}
\end{equation}
Then, integrating out the delta function with respect to $\tau$ and using the Plancherel theorem again,
\begin{equation}
\begin{aligned}
&\||\nabla|^{\tilde{\alpha}}\rho_{e^{it\Delta}\gamma_0 e^{-it\Delta}}\|_{L_{t\in\mathbb{R}}^2H_x^{\alpha_0}}^2\\
&\leq\sup_{\tau,\xi}\Big\{\int_{\mathbb{R}^d}\frac{|\xi|^{2\tilde{\alpha}}\la\xi\ra^{2\alpha_0}\delta(\tau+|\eta|^2-|\xi-\eta|^2)}{\la\eta\ra^{2\alpha_1}\la\xi-\eta\ra^{2\alpha_2}}d\eta\Big\} \frac{1}{(2\pi)^{2d}}\|\la\nabla\ra^{\alpha_1}\gamma_0\la\nabla\ra^{\alpha_2}\|_{\mathfrak{S}^2}^2.
\end{aligned}
\end{equation}
Therefore, it suffices to show that $\sup_{\tau,\xi}\{\cdots\}$ is bounded.

We decompose 
\begin{equation}
\begin{aligned}
&\int_{\mathbb{R}^d}\frac{|\xi|^{2\tilde{\alpha}}\la\xi\ra^{2\alpha_0}\delta(\tau+|\eta|^2-|\xi-\eta|^2)}{\la\eta\ra^{2\alpha_1}\la\xi-\eta\ra^{2\alpha_2}}d\eta\\
&=\int_{|\eta|\leq|\xi-\eta|}+\int_{|\eta|\geq|\xi-\eta|}\frac{|\xi|^{2\tilde{\alpha}}\la\xi\ra^{2\alpha_0}\delta(\tau+|\eta|^2-|\xi-\eta|^2)}{\la\eta\ra^{2\alpha_1}\la\xi-\eta\ra^{2\alpha_2}}d\eta.
\end{aligned}
\end{equation}
By change of the variable $(\xi-\eta)\mapsto\eta$, the second integral becomes
\begin{equation}
\begin{aligned}
&\int_{|\eta|\geq|\xi-\eta|}\frac{|\xi|^{2\tilde{\alpha}}\la\xi\ra^{2\alpha_0}\delta(\tau+|\eta|^2-|\xi-\eta|^2)}{\la\eta\ra^{2\alpha_1}\la\xi-\eta\ra^{2\alpha_2}}d\eta\\
&=\int_{|\eta|\leq|\xi-\eta|}\frac{|\xi|^{2\tilde{\alpha}}\la\xi\ra^{2\alpha_0}\delta(\tau+|\xi-\eta|^2-|\eta|^2)}{\la\eta\ra^{2\alpha_2}\la\xi-\eta\ra^{2\alpha_1}}d\eta\\
&=\int_{|\eta|\leq|\xi-\eta|}\frac{|\xi|^{2\tilde{\alpha}}\la\xi\ra^{2\alpha_0}\delta(-\tau+|\eta|^2-|\xi-\eta|^2)}{\la\eta\ra^{2\alpha_2}\la\xi-\eta\ra^{2\alpha_1}}d\eta.
\end{aligned}
\end{equation}
Thus, by the assumption \eqref{integral estimate}, we prove the desired uniform bound,
\begin{equation}
\int_{\mathbb{R}^d}\frac{|\xi|^{2\tilde{\alpha}}\la\xi\ra^{2\alpha_0}\delta(\tau+|\eta|^2-|\xi-\eta|^2)}{\la\eta\ra^{2\alpha_1}\la\xi-\eta\ra^{2\alpha_2}}d\eta\leq I_{\tau,\xi}+I_{-\tau,\xi}\leq 2\sup_{\tau,\xi}I_{\tau,\xi}<\infty.
\end{equation}
\end{proof}

\begin{proof}[Proof of Theorem \ref{Strichartz estimates for density functions in the non-relativistic case}]

By Lemma \ref{reduction lemma}, the proof of Theorem \ref{Strichartz estimates for density functions in the non-relativistic case} can be reduced to the proof of a uniform bound on the integral
\begin{equation}\label{Strichartz integral}
\begin{aligned}
I_{\tau,\xi}&=\int_{\{|\eta|\leq|\xi-\eta|\}}\frac{|\xi|\la\xi\ra^{2\alpha_0}\delta(\tau+|\eta|^2-|\xi-\eta|^2)}{\la\eta\ra^{2\max\{\alpha_1,\alpha_2\}}\la\xi-\eta\ra^{2\min\{\alpha_1,\alpha_2\}}} d\eta\\
&=\int_{\{|\eta|\leq|\xi-\eta|\}}\frac{|\xi|\la\xi\ra^{2\alpha_0}\delta(\tau-|\xi|^2+2\xi\cdot\eta)}{\la\eta\ra^{2\max\{\alpha_1,\alpha_2\}}\la\xi-\eta\ra^{2\min\{\alpha_1,\alpha_2\}}} d\eta.
\end{aligned}
\end{equation}
Here, we may assume that $\tau\geq 0$, since if $\tau<0$, then $\tau+|\eta|^2-|\xi-\eta|^2<0$ in the integral domain, so the delta function in \eqref{Strichartz integral} is zero. 

When $d=1$, using the trivial inequality
\begin{equation}\label{trivial inequality}
|\xi|\leq|\eta|+|\xi-\eta|\leq 2|\xi-\eta|
\end{equation}
in the integral domain, we obtain
$$I_{\tau,\xi}\lesssim\frac{\la\xi\ra^{2\alpha_0}}{\la\xi\ra^{\min\{\alpha_1,\alpha_2\}}}\int_{|\eta|\leq|\xi-\eta|}|\xi|\delta(\tau-\xi^2+2\xi\eta) d\eta\sim1.$$
Suppose that $d\geq 2$. Given $\xi\in\mathbb{R}^d$, changing the variable $\eta$ by a rotation making $(1,0,\cdots, 0)\in\mathbb{R}_\eta^d$ parallel to $\xi$ and then integrating out the delta function, we write the integral as
\begin{equation}
\begin{aligned}
I_{\tau,\xi}&=\int_{\mathbb{R}^{d-1}}\int_{|\eta_1|\leq|\eta_1-|\xi||}\frac{|\xi|\la\xi\ra^{2\alpha_0}\delta(\tau-|\xi|^2+2|\xi|\eta_1) }{\la(\eta_1,\eta')\ra^{2\max\{\alpha_1,\alpha_2\}}\la(\eta_1-|\xi|,\eta')\ra^{2\min\{\alpha_1,\alpha_2\}}}d\eta_1 d\eta'\\
&=\frac{1}{2}\int_{\mathbb{R}^{d-1}}\frac{\la\xi\ra^{2\alpha_0}d\eta'}{\la(\eta_1^*,\eta')\ra^{\max\{\alpha_1,\alpha_2\}}\la(\eta_1^*-|\xi|,\eta')\ra^{2\min\{\alpha_1,\alpha_2\}}},
\end{aligned}
\end{equation}
where $\eta=(\eta_1,\eta')\in\mathbb{R}\times\mathbb{R}^{d-1}$ and $\eta_1^*=\tfrac{|\xi|^2-\tau}{2|\xi|}$ with $|\eta_1^*|\leq|\eta_1^*-|\xi||$. Note that by the trivial inequality as in \eqref{trivial inequality}, we have $|\eta_1^*-|\xi||\geq\frac{|\xi|}{2}$. Thus, Theorem \ref{Strichartz estimates for density functions in the non-relativistic case} follows from the uniform bound on
$$\tilde{I}_{\tau,\xi}:=\int_{\mathbb{R}^{d-1}}\frac{\la\xi\ra^{2\alpha_0}d\eta'}{\la \eta'\ra^{2\max\{\alpha_1,\alpha_2\}}\la(\frac{|\xi|}{2},\eta')\ra^{2\min\{\alpha_1,\alpha_2\}}}.$$

We decompose
\begin{equation}
\tilde{I}_{\tau,\xi}=\int_{|\eta'|\leq|\xi|}+\int_{|\eta'|\geq|\xi|}\frac{\la\xi\ra^{2\alpha_0}d\eta'}{\la \eta'\ra^{2\max\{\alpha_1,\alpha_2\}}\la(\frac{|\xi|}{2},\eta')\ra^{2\min\{\alpha_1,\alpha_2\}}}=:\tilde{I}_{\tau,\xi}^{(1)}+\tilde{I}_{\tau,\xi}^{(2)}.
\end{equation}
For the first integral, using that $\frac{|\xi|}{2}\leq|(\frac{|\xi|}{2},\eta')|\leq \frac{\sqrt{5}}{2}|\xi|$ in the integral domain, we get
\begin{equation}
\begin{aligned}
\tilde{I}_{\tau,\xi}^{(1)}&\sim\int_{|\eta'|\leq|\xi|}\frac{\la\xi\ra^{2\alpha-2\min\{\alpha_1,\alpha_2\}}d\eta'}{\la \eta'\ra^{2\max\{\alpha_1,\alpha_2\}}}\\
&\sim\left\{\begin{aligned}
&\la\xi\ra^{2\alpha_0-2(\alpha_1+\alpha_2)+d-1}&&\textup{if }0\leq\max\{\alpha_1,\alpha_2\}<\tfrac{d-1}{2},\\
&\la\xi\ra^{2\alpha_0-2\min\{\alpha_1,\alpha_2\}}\ln\la\xi\ra&&\textup{if }\max\{\alpha_1,\alpha_2\}=\tfrac{d-1}{2},\\
&\la\xi\ra^{2\alpha_0-2\min\{\alpha_1,\alpha_2\}}&&\textup{if }\max\{\alpha_1,\alpha_2\}>\tfrac{d-1}{2}.
\end{aligned}
\right.
\end{aligned}
\end{equation}
The second integral $\tilde{I}_{\tau,\xi}^{(2)}$ is bounded by
\begin{equation}
\int_{|\eta'|\geq|\xi|}\frac{\la\xi\ra^{2\alpha_0}d\eta'}{\la \eta'\ra^{2\max\{\alpha_1,\alpha_2\}+2\min\{\alpha_1,\alpha_2\}}}=\int_{|\eta'|\geq|\xi|}\frac{\la\xi\ra^{2\alpha_0}d\eta'}{\la \eta'\ra^{2(\alpha_1+\alpha_2)}}\lesssim\la\xi\ra^{2\alpha_0-2(\alpha_1+\alpha_2)+(d-1)},
\end{equation}
since $2(\alpha_1+\alpha_2)>d-1$. Both $\tilde{I}_{\tau,\xi}^{(1)}$ and $\tilde{I}_{\tau,\xi}^{(2)}$ are uniformly bounded due to the assumption \eqref{non-relativistic beta}.
\end{proof}

Next, we prove optimality of the Strichartz estimate \eqref{non-relativistic Strichartz}.
\begin{theorem}[Optimality of Theorem \ref{Strichartz estimates for density functions in the non-relativistic case}]\label{optimality}
The assumptions in Theorem \ref{Strichartz estimates for density functions in the non-relativistic case} are necessary.
\end{theorem}

The following dual formulation is useful to find the necessary conditions on the Strichartz estimate \eqref{non-relativistic Strichartz}.
\begin{lemma}[Dual inequality]\label{duality lemma}
The Strichartz estimate \eqref{Strichartz estimate: reduction lemma} holds if and only if
\begin{equation}\label{dual Strichartz estimate: reduction lemma}
\Big\|\frac{|\xi|^{\tilde{\alpha}}\la\xi\ra^{\alpha_0}\tilde{V}(-|\eta|^2+|\xi-\eta|^2,\xi)}{\la\eta\ra^{\alpha_1}\la\xi-\eta\ra^{\alpha_2}}\Big\|_{L_{\xi}^2L_\eta^2}\leq \|\tilde{V}(\tau,\xi)\|_{L_{\tau\in\mathbb{R}}^2 L_\xi^2}.
\end{equation}
\end{lemma}

\begin{proof}
Using the Plancherel theorem and \eqref{space-time Fourier transform of density} and then integrating out the delta function, we write
\begin{equation}
\begin{aligned}
&\int_{\mathbb{R}}\int_{\mathbb{R}^d} (|\nabla|^{\tilde{\alpha}}\la\nabla\ra^{\alpha_0}\rho_{e^{it\Delta}\gamma_0e^{-it\Delta}})(x) \overline{V(t,x)} dxdt\\
&=\frac{1}{(2\pi)^{d+1}}\int_{\mathbb{R}}\int_{\mathbb{R}^d} \Big\{\frac{1}{(2\pi)^{d-1}}\int_{\mathbb{R}^d}\delta(\tau+|\eta|^2-|\xi-\eta|^2)\hat{\gamma}_0(\eta,\xi-\eta)d\eta\Big\}\\
&\quad\quad\quad\quad\quad\quad\quad\quad\cdot\overline{|\xi|^{\tilde{\alpha}}\la\xi\ra^{\alpha_0}\tilde{V}(\tau,\xi)} d\xi d\tau\\
&=\frac{1}{(2\pi)^{2d}}\int_{\mathbb{R}^d} \int_{\mathbb{R}^d}\hat{\gamma}_0(\eta,\xi-\eta)\overline{|\xi|^{\tilde{\alpha}}\la\xi\ra^{\alpha_0}\tilde{V}(-|\eta|^2+|\xi-\eta|^2,\xi)} d\eta d\xi.
\end{aligned}
\end{equation}
By H\"older inequality and the Plancherel theorem, it is bounded by
\begin{equation}
\begin{aligned}
&\frac{1}{(2\pi)^{2d}}\|(\la\nabla\ra^{\alpha_1}\gamma_0\la\nabla\ra^{\alpha_2})^\wedge\|_{L_{\xi,\eta}^2}\Big\|\frac{|\xi|^{\tilde{\alpha}}\la\xi\ra^{\alpha_0}\tilde{V}(-|\eta|^2+|\xi-\eta|^2,\xi)}{\la\eta\ra^{\alpha_1}\la\xi-\eta\ra^{\alpha_2}}\Big\|_{L_{\xi,\eta}^2}\\
&=\frac{1}{(2\pi)^{d}}\|\la\nabla\ra^{\alpha_1}\gamma_0\la\nabla\ra^{\alpha_2}\|_{\mathfrak{S}^2}\Big\|\frac{|\xi|^{\tilde{\alpha}}\la\xi\ra^{\alpha_0}\tilde{V}(-|\eta|^2+|\xi-\eta|^2,\xi)}{\la\eta\ra^{\alpha_1}\la\xi-\eta\ra^{\alpha_2}}\Big\|_{L_{\xi,\eta}^2}.
\end{aligned}
\end{equation}
Therefore, by duality, \eqref{Strichartz estimate: reduction lemma} is equivalent to \eqref{dual Strichartz estimate: reduction lemma}.
\end{proof}

\begin{proof}[Proof of Theorem \ref{optimality}]
By the duality lemma (Lemma \ref{duality lemma}), the inequality \eqref{non-relativistic Strichartz} holds if and only if
\begin{equation}\label{non-relativistic dual inequality}
\Big\|\frac{|\xi|^{\tilde{\alpha}}\la\xi\ra^{\alpha_0}\tilde{V}(-|\eta|^2+|\xi-\eta|^2,\xi)}{\la\eta\ra^{\alpha_1}\la\xi-\eta\ra^{\alpha_2}}\Big\|_{L_{\xi}^2L_\eta^2}\lesssim \|\tilde{V}(\tau,\xi)\|_{L_{\tau\in\mathbb{R}}^2 L_\xi^2}.
\end{equation}
The square of the left hand side is
\begin{equation}
\begin{aligned}
&\Big\|\frac{|\xi|^{\tilde{\alpha}}\la\xi\ra^{\alpha_0}\tilde{V}(-|\eta|^2+|\xi-\eta|^2,\xi)}{\la\eta\ra^{\alpha_1}\la\xi-\eta\ra^{\alpha_2}}\Big\|_{L_{\xi}^2L_\eta^2}^2\\
&=\int_{\mathbb{R}^d}\int_{\mathbb{R}^d}\frac{|\xi|^{2\tilde{\alpha}}\la\xi\ra^{2\alpha_0}}{\la\eta\ra^{2\alpha_1}\la\xi-\eta\ra^{2\alpha_2}}|\tilde{V}(|\xi|^2-2\xi\cdot\eta, \xi)|^2 d\eta d\xi.
\end{aligned}
\end{equation}
Changing the variable $\eta$ by a rotation making $(1,0,\cdots, 0)\in\mathbb{R}_\eta^d$ parallel to $\xi$ and then changing the variable $\tau=|\xi|^2-2|\xi|\eta_1$ as in the proof of Theorem \ref{Strichartz estimates for density functions in the non-relativistic case}, we write
\begin{equation}\label{non-relativistic dual reformulation}
\begin{aligned}
&\int_{\mathbb{R}^d}\int_{\mathbb{R}^{d-1}}\int_{\mathbb{R}}\frac{|\xi|^{2\tilde{\alpha}}\la\xi\ra^{2\alpha_0}}{\la(\eta_1,\eta')\ra^{2\alpha_1}\la(\eta_1-|\xi|,\eta')\ra^{2\alpha_2}}|\tilde{V}(|\xi|^2-2|\xi|\eta_1, \xi)|^2 d\eta_1 d\eta' d\xi\\
&=\frac{1}{2}\int_{\mathbb{R}^d}\int_{\mathbb{R}^{d-1}}\int_{\mathbb{R}}\frac{|\xi|^{2\tilde{\alpha}-1}\la\xi\ra^{2\alpha_0}}{\la(\frac{|\xi|^2-\tau}{2|\xi|},\eta')\ra^{2\alpha_1}\la(-\frac{|\xi|^2+\tau}{2|\xi|},\eta')\ra^{2\alpha_2}}|\tilde{V}(\tau, \xi)|^2 d\tau d\eta' d\xi\\
&=\frac{1}{2}\int_{\mathbb{R}}\int_{\mathbb{R}^d}\Big\{\int_{\mathbb{R}^{d-1}}\frac{|\xi|^{2\tilde{\alpha}-1}\la\xi\ra^{2\alpha_0}}{\la(\frac{|\xi|^2-\tau}{2|\xi|},\eta')\ra^{2\alpha_1}\la(-\frac{|\xi|^2+\tau}{2|\xi|},\eta')\ra^{2\alpha_2}}d\eta'\Big\} |\tilde{V}(\tau, \xi)|^2 d\xi d\tau,
\end{aligned}
\end{equation}
where $\eta=(\eta_1,\eta')\in\mathbb{R}\times\mathbb{R}^d$.

1. Necessity of the condition \eqref{non-relativistic alpha}: From the inner integral $\{\cdots\}$ over $\mathbb{R}^{d-1}$ in \eqref{non-relativistic dual reformulation}, we see that it is necessary to assume that $\alpha_1+\alpha_2>\frac{d-1}{2}$ in \eqref{non-relativistic alpha}, because if $\alpha_1+\alpha_2\leq\frac{d-1}{2}$, the inequality \eqref{non-relativistic Strichartz} fails.

2. Necessity of the homogeneous half derivative on the left hand side of \eqref{non-relativistic Strichartz}: Suppose that $\alpha_1+\alpha_2>\frac{d-1}{2}$. Let
$$\tilde{V}_n(\tau,\xi)=n^{\frac{d+2}{2}}\mathbf{1}_{[-\frac{1}{n^2},\frac{1}{n^2}]}(\tau)\mathbf{1}_{B_{0,\frac{1}{n}}}(\xi),$$
where $B_{0,r}$ is the ball of radius $r$ centered at $0$ in $\mathbb{R}^d$. Note that for large $n$, $\tilde{V}_n$ is localized in low frequencies. We observe that by \eqref{non-relativistic dual reformulation}, if $\tilde{\alpha}<\frac{1}{2}$, then 
\begin{equation}
\begin{aligned}
&\Big\|\frac{|\xi|^{\tilde{\alpha}}\la\xi\ra^{\alpha_0}\tilde{V}(-|\eta|^2+|\xi-\eta|^2,\xi)}{\la\eta\ra^{\alpha_1}\la\xi-\eta\ra^{\alpha_2}}\Big\|_{L_{\xi}^2L_\eta^2}^2\\
&\sim\int_{\mathbb{R}}\int_{\mathbb{R}^d}\Big\{\int_{\mathbb{R}^{d-1}}\frac{n^{1-2\tilde{\alpha}}}{\la \eta' \ra^{2(\alpha_1+\alpha_2)}}d\eta'\Big\} |\tilde{V}(\tau, \xi)|^2 d\xi d\tau\sim n^{1-2\tilde{\alpha}}\underset{n\to\infty}\longrightarrow\infty,
\end{aligned}
\end{equation}
while $\|\tilde{V}_n\|_{L_{\tau\in\mathbb{R}}^2L_\xi^2}\sim 1$.
Thus, the inequality \eqref{non-relativistic Strichartz} fails when $\tilde{\alpha}<\frac{1}{2}$. 

3. Necessity of the condition \eqref{non-relativistic beta}: Suppose that $\alpha_1+\alpha_2>\frac{d-1}{2}$ and $\tilde{\alpha}=\frac{1}{2}$. We further assume that $\alpha_1\geq\alpha_2$. Now we define the sequence $\{V_n\}_{n=1}^\infty$ by
\begin{equation}
\tilde{V}_n(\tau,\xi)=\mathbf{1}_{[n^2-\frac{1}{2},n^2+\frac{1}{2}]}(\tau)\mathbf{1}_{[n-\frac{1}{2},n+\frac{1}{2}]\times[\frac{1}{2},\frac{1}{2}]^{d-1}}(\xi),
\end{equation}
where $\xi=(\xi_1,\xi')\in\mathbb{R}\times\mathbb{R}^{d-1}$, so that $\|\tilde{V}_n\|_{L_{\tau}^2L_\xi^2}=1$. Then, $|\frac{|\xi|^2-\tau}{2|\xi|}|\leq\frac{1}{2}+o_n(1)$ and $-\frac{|\xi|^2+\tau}{2|\xi|}=-n+o_n(1)$ in the support of $\tilde{V}(\tau,\xi)$, where $o_n(1)\to 0$ as $n\to\infty$. Hence, by \eqref{non-relativistic dual reformulation},
\begin{equation}
\begin{aligned}
&\Big\|\frac{|\xi|^{1/2}\la\xi\ra^{\alpha_0}\tilde{V}_n(-|\eta|^2+|\xi-\eta|^2,\xi)}{\la\eta\ra^{\alpha_1}\la\xi-\eta\ra^{\alpha_2}}\Big\|_{L_{\xi}^2L_\eta^2}^2\\
&\gtrsim \int_{|\eta'|\leq\frac{n}{2}}\frac{n^{2\alpha_0}}{\la\eta'\ra^{2\alpha_1}\la(-n,\eta')\ra^{2\alpha_2}}d\eta'\sim n^{2\alpha_0-2\alpha_2}\int_{|\eta'|\leq\frac{n}{2}}\frac{d\eta'}{\la\eta'\ra^{2\alpha_1}}\\
&\sim\left\{\begin{aligned}
&n^{2\alpha_0-2(\alpha_1+\alpha_2)+d-1}&&\textup{if }0\leq\alpha_1<\tfrac{d-1}{2},\\
&n^{2\alpha_0-2\alpha_2}\ln n&&\textup{if }\alpha_1=\tfrac{d-1}{2},\\
&n^{2\alpha_0-2\alpha_2}&&\textup{if }\alpha_1>\tfrac{d-1}{2}
\end{aligned}
\right.
\end{aligned}
\end{equation}
for sufficiently large $n$. Thus, \eqref{non-relativistic dual inequality} fails unless \eqref{non-relativistic beta} is not satisfied.

When $\alpha_1\leq\alpha_2$, we use the sequence $\{V_n\}_{n=1}^\infty$ given by
\begin{equation}
\tilde{V}(\tau,\xi)=\mathbf{1}_{[-n^2-\frac{1}{2},-n^2+\frac{1}{2}]}(\tau)\mathbf{1}_{[n-\frac{1}{2},n+\frac{1}{2}]\times[\frac{1}{2},\frac{1}{2}]^{d-1}}(\xi)
\end{equation}
to prove that the condition \eqref{non-relativistic beta} is necessary.
\end{proof}

\subsection{Strichartz estimates for operator kernels} \label{subs3-Str} 

We finish this section by recalling the statement of the Strichartz estimates for operator kernels, 
that we established in  \cite{CHP1}.

\begin{theorem}[Strichartz estimates for operator kernels]\label{Strichartz estimates for operator kernels}
Let $I\subset\mathbb{R}$. Then, we have
\begin{equation}
\begin{aligned}
\|e^{it\Delta}\gamma_0e^{-it\Delta}\|_{\mathcal{S}^\alpha(\mathbb{R})}&\lesssim\|\gamma_0\|_{\mathcal{H}^\alpha},\\
\Big\|\int_0^t e^{i(t-s)\Delta}R(s)e^{-i(t-s)\Delta}ds\Big\|_{\mathcal{S}^\alpha(\mathbb{R})}&\lesssim\|R(t)\|_{L_t^1(\mathbb{R};\mathcal{H}^\alpha)}.
\end{aligned}
\end{equation}
\end{theorem}

\section{Linear response theory: invertibility of $(1+\mathcal{L})$}\label{invertibility}
% Replace (1-L) by (1+L)

We review the linear response theory from Section 3 of Lewin and Sabin \cite{LS2}, which addresses the invertibility of the operator $(1+\mathcal{L})$, 
with $\mathcal{L}$ defined by 
\begin{equation}
\begin{aligned}
\mathcal{L}(\phi)&=-w_2*\rho\Big[e^{it\Delta}\Big(\mathcal{W}_{w_1*\phi}^{(1)}(t)\gamma_f+\gamma_f\mathcal{W}_{w_1*\phi}^{(1)}(t)^*\Big) e^{-it\Delta}\Big]\\
&=iw_2*\rho\Big[\int_0^t e^{i(t-t_1)\Delta}\big[(w_1*\phi)(t_1),\gamma_f\big]e^{-i(t-t_1)\Delta} dt_1\Big],
\end{aligned}
\end{equation}
where $w=w_1*w_2$. Roughly speaking, it asserts that $(1+\mathcal{L})$ is invertible on $L_{t\geq 0}^2L_x^2$, provided that $f$ is strictly decreasing, and that $\hat{w}_+(0)$ and $\hat{w}_-$ are not too large, where $A_\pm=\max\{\pm A,0\}$ so that $A=A_+-A_-$.

\begin{proposition}[Invertibility of $(1+\mathcal{L})$]\label{prop: invertibility}
Let $d\geq 3$. We assume that $f\in L_{r\geq 0}^\infty$ is real-valued, $f'(r)<0$ for $r>0$,
\begin{equation}\label{condition on f}
\int_0^\infty (r^{d/2-1}|f(r)|+|f'(r)|)dr<\infty\quad \textup{and}\quad \int_{\mathbb{R}^d} \frac{\check{g}(x)}{|x|^{d-2}}dx<\infty,
\end{equation}
where $g(\xi)=f(|\xi|^2)$. Moreover, we assume that the interaction potential $w\in L^1$ is even,
\begin{equation} \label{sec4:w-assum}
\|\hat{w}_-\|_{L^\infty}<2|\mathbb{S}^{d-1}|\Big(\int_{\mathbb{R}^d}\frac{|\check{g}(x)|}{|x|^{d-2}}dx\Big)^{-1}\quad\textup{and}\quad
\hat{w}_+(0)<\frac{2}{\epsilon_g}|\mathbb{S}^{d-1}|,
\end{equation}
where
\begin{equation}\label{e_g}
\epsilon_g:=-\liminf_{(\tau,\xi)\to(0,0)}\frac{\textup{Re}(m_f(\tau,\xi))}{2|\mathbb{S}^{d-1}|}
\end{equation}
and
\begin{equation}\label{m_f}
(\mathcal{F}_t^{-1}m_f)(t,\xi)=2\mathbf{1}_{t\geq 0}\sqrt{2\pi}\sin(t|\xi|^2)\check{g}(2t\xi).
\end{equation}
Then, $1+\mathcal{L}$ is invertible on $L_{t\geq0 }^2L_x^2$.
\end{proposition}

\begin{proof}[Sketch of the proof]
We sketch the proof for the sake of completeness of the article and for the convenience of the reader. For details, we refer the reader to \cite[Proposition 1, Proposition 2 and Corollary 1]{LS2}. We assume $d\geq 3$ for brevity, however, the invertibility of $(1+\mathcal{L})$ was proved in \cite{LS2} for any dimension $d\geq 1$.

The space-time Fourier transformation of $\mathcal{L}(\phi)$ is directly computed as
\begin{equation}
(\mathcal{L}\phi)^\sim(\tau,\xi)=\hat{w}(\xi)m_f(\tau,\xi)\tilde{\phi}(\tau,\xi),\quad \forall\phi\in \mathcal{D}([0,+\infty)\times\mathbb{R}^d),
\end{equation}
with \eqref{m_f}, in other words,
\begin{equation}\label{explicit L}
\widehat{\mathcal{L}\phi}(t,\xi)=2\sqrt{2\pi}\hat{w}(\xi)\int_0^\infty \sin(s|\xi|^2)\check{g}(2s\xi)\hat{\phi}(t-s,\xi)ds.
\end{equation}
Note that the operator $\mathcal{L}$ maps $L_{t\geq0}^2L_x^2$ to itself, because $\widehat{\mathcal{L}\phi}(t,\xi)=0$ for $t<0$. Moreover, we have
\begin{equation}\label{m_f bound}
\|m_f\|_{L_{\tau,\xi}^\infty}\leq\frac{1}{2|\mathbb{S}^{d-1}|}\Big(\int_{\mathbb{R}^d}\frac{|\check{g}(x)|}{|x|^{d-2}}dx\Big)
\end{equation}
and
\begin{equation}
\|\mathcal{L}\|_{L_{t\geq0}^2L_x^2\to L_{t\geq 0}^2L_x^2}\leq\frac{\|\hat{w}\|_{L^\infty}}{2|\mathbb{S}^{d-1}|}\Big(\int_{\mathbb{R}^d}\frac{|\check{g}(x)|}{|x|^{d-2}}dx\Big)
\end{equation}
(see \cite[Proposition 1]{LS2}). We remark that the operator $\mathcal{L}$ looks different from the corresponding linear operator $\mathcal{L}_1$ in Lewin-Sabin \cite{LS2} at first glance, however they are indeed the same, since
\begin{equation}
\begin{aligned}
(\mathcal{L}\phi)^\sim(\tau,\xi)&=\hat{w}_2(\xi)\Big(\rho\Big[i\int_0^t e^{i(t-t_1)\Delta}\big[(w_1*\phi)(t_1),\gamma_f\big]e^{-i(t-t_1)\Delta} dt_1\Big]\Big)^\sim(\tau,\xi)\\
&=\hat{w}_2(\xi)\hat{w}_1(\xi)m_f(\tau,\xi)\tilde{\phi}(\tau,\xi)\quad\textup{(by \cite[Proposition 1]{LS2})}\\
&=\hat{w}(\xi)m_f(\tau,\xi)\tilde{\phi}(\tau,\xi).
\end{aligned}
\end{equation}
When $\gamma_f=\mathbf{1}_{(-\Delta\leq\mu)}$, one can compute the multiplier $m_d^F(\mu,\tau,\xi):=m_f(\tau,\xi)$ as
\begin{equation}\label{m_d^F}
m_d^F(\mu,\tau,\xi)=\frac{|\mathbb{S}^{d-2}|\mu^{\frac{d-1}{2}}}{(2\pi)^{\frac{d-1}{2}}}\int_0^1 m_1^F(\mu(1-r^2),\tau,\xi) r^{d-2}dr,
\end{equation}
where
\begin{equation}\label{m_1^F}
\begin{aligned}
m_1^F(\mu,\tau,\xi)&=\frac{1}{2\sqrt{2\pi}|\xi|}\log\Big|\frac{(|\xi|^2+2|\xi|\sqrt{\mu})^2-\tau^2}{(|\xi|^2-2|\xi|\sqrt{\mu})^2-\tau^2}\Big|\\
&+i\frac{\sqrt{\pi}}{2\sqrt{2}|\xi|}\Big\{\mathbf{1}_{(|\tau+|\xi|^2|\leq 2\sqrt{\mu}|\xi|)}-\mathbf{1}_{(|\tau-|\xi|^2|\leq 2\sqrt{\mu}|\xi|)}\Big\}
\end{aligned}
\end{equation}
(see \cite[Proposition 2]{LS2}). By the relation $\gamma_f=f(-\Delta)=-\int_0^\infty \mathbf{1}_{(-\Delta\leq s)}f'(s)ds$, $m_f$ can be written in terms of $m_d^F$ as 
\begin{equation}\label{m_f'}
m_f(\tau,\xi)=-\int_0^\infty m_d^F(s,\tau,\xi) f'(s)ds.
\end{equation}

For $\phi\in L_{t\geq0}^2L_x^2$, the space-time Fourier transformation of $(1+\mathcal{L})\phi$ is given by $(1+\hat{w}(\xi)m_f(\tau,\xi))\tilde{\phi}(\tau,\xi)$. Thus, the invertibility of $(1+\mathcal{L})$ follows from a uniform lower bound on $|1+\hat{w}m_f|$. Let
\begin{equation}
A:=\Big\{\xi\in\mathbb{R}^d:\ |\hat{w}(\xi)|\geq\frac{1}{4|\mathbb{S}^{d-1}|}\int_{\mathbb{R}^d}\frac{|\check{g}(x)|}{|x|^{d-2}}dx\Big\}.
\end{equation}
Then, by the bound \eqref{m_f bound}, $|(1+\hat{w}m_f)|\geq\frac{1}{2}$ on $A$. Note that $A^c$ is a compact subset in $\mathbb{R}^d$, because $\hat{w}(\xi)\to 0$ as $\xi\to\infty$. Moreover, by \eqref{m_f'}, $m_f$ is continuous on $\mathbb{R}\times(\mathbb{R}^d\setminus\{0\})$ (so as $(1+\hat{w}m_f)$ by the Riemann-Lebesgue lemma), since $m_d^F$ is continuous on $\mathbb{R}\times(\mathbb{R}^d\setminus\{0\})$. Therefore, it suffices to show that $(1+\hat{w}m_f)$ is non-zero for all $\xi$.

We consider the four cases separately.\\
\textbf{Case 1 ($(\tau,\xi)=(0,\xi)$ with $\xi\neq0$)} We observe that $m_f(0,\xi)\geq0$ for $\xi\neq0$, since $f'(s)< 0$ and $m_1^F(s,0,\xi)\geq0$ in the integral \eqref{m_f'} (see \eqref{m_1^F}). Hence, it follows that
\begin{equation}
\begin{aligned}
m_f(0,\xi)\hat{w}(\xi)+1&\geq 1-\hat{w}_-(\xi)m_f(0,\xi)\\
&\geq 1-\|\hat{w}_-\|_{L^\infty}\frac{1}{2|\mathbb{S}^{d-1}|}\Big(\int_{\mathbb{R}^d}\frac{|\check{g}(x)|}{|x|^{d-2}}dx\Big)\quad\textup{(by \eqref{m_f bound})}\\
&>0\quad\textup{(by the assumption \eqref{sec4:w-assum} on $\hat{w}_-$)}.
\end{aligned}
\end{equation}
\textbf{Case 2 ($(\tau,\xi)=(\tau,0)$ with $\tau\neq0$)} In this case, $m_1^F(\tau,0)=0$, so $(m_f(\tau,0)\hat{w}(0)+1)=1$.\\
\textbf{Case 3 ($(\tau,\xi)$ with $\tau\neq0$ and $\xi\neq0$)} It suffices to show that $\textup{Im}(m_f(\tau,\xi))\neq0$. By the relation $\textup{Im}(m_f(-\tau,\xi))=-\textup{Im}(m_f(\tau,\xi))$, we may assume that $\tau>0$. By \eqref{m_f'} and \eqref{m_d^F}, one can write the imaginary part of $m_f(\tau,\xi)$ explicitly as
\begin{equation}
\begin{aligned} \label{sec4:case3}
\textup{Im}(m_f(\tau,\xi))&=\frac{|\mathbb{S}^{d-2}|}{4(2\pi)^{\frac{d-2}{2}}}\int_0^1r^{d-2}\Big\{\int_{\frac{(\tau-|\xi|^2)^2}{4|\xi|^2(1-r^2)}}^{\frac{(\tau+|\xi|^2)^2}{4|\xi|^2(1-r^2)}} s^{\frac{d-1}{2}}f'(s)ds\Big\}dr,
\end{aligned}
\end{equation}
Since by the assumption $f'(s)<0$, we conclude from \eqref{sec4:case3} that  $\textup{Im}(m_f(\tau,\xi))\neq0$.\\
\textbf{Case 4 ($(\tau,\xi)$ in the neighborhood of $(0,0)$)} By the definition of $m_f$ and \eqref{m_f'}, one can show that
\begin{equation}
-\epsilon_g2|\mathbb{S}^{d-1}|\leq \textup{Re}(m_f(\tau,\xi))\leq\frac{1}{2|\mathbb{S}^{d-1}|}\Big(\int_{\mathbb{R}^d}\frac{|\check{g}(x)|}{|x|^{d-2}}dx\Big)
\end{equation}
near $(0,0)$ (see \cite{LS2} for details). Thus, by the assumptions on $\hat{w}_\pm$, $\textup{Re}(\hat{w}(\xi)m_f(\tau,\xi)+1)>0$.
\end{proof}

\section{Bound on $\mathcal{A}_{m,n}(\phi)$}\label{bound on A}

In this section, we estimate the operator $\mathcal{A}_{m,n}$.
\begin{proposition}[Bounds on $\mathcal{A}_{m,n}$]\label{bound on A}
Let $d\geq 3$, $\beta>\frac{d+2}{2}$ and $\beta_0>\frac{1}{4}$. Then, there exists $C_{\mathcal{A}}>0$ such that for any $(m,n)$ with $m,n\geq 1$, 
\begin{equation}\label{bound on A1}
\|\mathcal{A}_{m,n}(\phi)\|_{L_{t,x}^2}\leq C_{\mathcal{A}}^{m+n+1}\|\la\cdot\ra^\beta f\|_{L^\infty}\|\phi\|_{L_{t,x}^2}^{m+n}
\end{equation}
and
\begin{equation}\label{bound on A2}
\begin{aligned}
&\|\mathcal{A}_{m,n}(\phi)-\mathcal{A}_{m,n}(\psi)\|_{L_{t,x}^2}\\
&\leq (m+n)C_{\mathcal{A}}^{m+n+1}\|\la\cdot\ra^\beta f\|_{L^\infty}\Big\{\|\phi\|_{L_{t,x}^2}+\|\psi\|_{L_{t,x}^2}\Big\}^{m+n-1}\|\phi-\psi\|_{L_{t,x}^2},
\end{aligned}
\end{equation}
where the constant $C_{\mathcal{A}}$ depends only on $d$, $\|\langle\cdot\rangle^{\beta_0}\hat{w}_1\|_{L^\infty}$, $\|\hat{w}_2\|_{L^\infty}$,  $\|\hat{w}_1\|_{L^{\frac{2d}{d-2}}}$, $\|\hat{w}_2\|_{L^{\frac{2d}{d-2}}}$ and $\|\hat{w}_2\|_{L^{\frac{2d}{d-3}}}$.
\end{proposition}

\begin{proof}
We will prove the proposition by the standard duality argument. For notational convenience, we denote $W=w_1*\phi$. By the definition of $\mathcal{A}_{m,n}$, we write
\begin{equation}
\begin{aligned}
&\int_0^\infty\int_{\mathbb{R}^d}\mathcal{A}_{m,n}(\phi)(t) U(t,x)dxdt\\
&=\int_0^\infty\int_{\mathbb{R}^d}w_2*\rho\Big[e^{it\Delta}\mathcal{W}_{W}^{(m)}(t)\gamma_f\mathcal{W}_{W}^{(n)}(t)^* e^{-it\Delta}\Big]U(t,x)dxdt\\
&=\int_0^\infty\int_{\mathbb{R}^d}\rho\Big[e^{it\Delta}\mathcal{W}_{W}^{(m)}(t)\gamma_f\mathcal{W}_{W}^{(n)}(t)^* e^{-it\Delta}\Big](w_2*U)(t,x)dxdt.
\end{aligned}
\end{equation}
Then, by the formal identity 
\begin{equation}\label{rho to trace}
\int_{\mathbb{R}^d}\rho_{\gamma_0}V dx=\textup{Tr}(\gamma_0 V)
\end{equation}
and the cyclicity of the trace, it becomes
\begin{equation}\label{rho to trace}
\begin{aligned}
&\int_0^\infty\int_{\mathbb{R}^d}\mathcal{A}_{m,n}(\phi)(t) U(t,x)dxdt\\
&=\textup{Tr}\Big(\int_0^\infty e^{it\Delta}\mathcal{W}_{W}^{(m)}(t)\gamma_f\mathcal{W}_{W}^{(n)}(t)^* e^{-it\Delta}(w_2*U)(t)dt\Big)\\
&=\textup{Tr}\Big(\int_0^\infty\mathcal{W}_{W}^{(m)}(t)\gamma_f\mathcal{W}_{W}^{(n)}(t)^* e^{-it\Delta}(w_2*U)(t)e^{it\Delta}dt\Big).
\end{aligned}
\end{equation}
Note that the application of the formal identity \eqref{rho to trace} in \eqref{rho to trace application} will be justified by the estimates below.

First, we consider the higher order terms with $m+n\geq 3$. In this case, we employ the following two inequalities, 
\begin{align}
\Big\|\int_0^\infty e^{-it\Delta}V(t)e^{it\Delta}dt\Big\|_{\mathfrak{S}^{2d}}&\leq c\|V\|_{L_t^2L_x^d},\label{interpolation1}\\
\Big\|\int_0^\infty e^{-it\Delta}V(t)e^{it\Delta}dt\langle\nabla\rangle^{-\tilde{\beta}}\Big\|_{\mathfrak{S}^{\frac{2d}{d-1}}}&\leq c\|V\|_{L_t^2L_x^2},\label{interpolation2}
\end{align}
where $\tilde{\beta}=\beta-2>\frac{d-2}{2}$. Here, \eqref{interpolation1} is from Theorem 2 in \cite{FLLS2} and \eqref{interpolation2} can be obtained from the complex interpolation between \eqref{interpolation1} and \eqref{eq: Bound on IV} with $\alpha_1=0$ and $\alpha_2>\frac{d-1}{2}$. Expanding $\mathcal{W}_W^{(m)}(t)$ and $\mathcal{W}_W^{(n)}(t)$ in the expression \eqref{rho to trace} (see \eqref{W_V^n}) and applying the inequality $|\textup{Tr}(AB)|\leq\textup{Tr}(|A||B|)$, we write
\begin{equation}\label{rho to trace application}
\begin{aligned}
&\Big|\int_0^\infty\int_{\mathbb{R}^d}\mathcal{A}_{m,n}(\phi)(t) U(t,x)dxdt\Big|\\
&\leq \textup{Tr}\Big\{\int_0^\infty\Big(\int_0^\infty\cdots \int _0^\infty e^{-it_m\Delta}|W(t_m)|e^{it_m\Delta}\cdots e^{-it_1\Delta}|W(t_1)|e^{it_1\Delta}dt_1\cdots dt_m\Big)\gamma_f\\
&\quad\quad\quad\quad\quad\Big(\int_0^\infty \cdots \int_0^\infty e^{-it_1'\Delta}|W(t_1')|e^{it_1'\Delta} \cdots e^{-it_n'\Delta}|W(t_n')|e^{it_n'\Delta}dt_1'\cdots dt_n'\Big)\\
&\quad\quad\quad\quad\quad\quad e^{-it\Delta}|w_1*U(t)|e^{it\Delta}dt\Big\}\\
&=\textup{Tr}\Big\{\Big(\int_0^\infty e^{-it_1\Delta}|W(t_1)|e^{it_1\Delta}dt_1\Big)\cdots \Big(\int_0^\infty e^{-it_m\Delta}|W(t_m)|e^{it_m\Delta}dt_m\Big)\gamma_f\\
&\quad\quad\quad\cdot\Big(\int_0^\infty e^{-it_n'\Delta}|W(t_n')|e^{it_n'\Delta}dt_n'\Big)\cdots \Big(\int_0^\infty e^{-it_1'\Delta}|W(t_1')|e^{it_1'\Delta}dt_1'\Big)\\
&\quad\quad\quad\cdot\Big(\int_0^\infty e^{-it\Delta}|w_1*U(t)|e^{it\Delta}dt\Big)\Big\}.
\end{aligned}
\end{equation}
When $m,n\geq 1$,  by the H\"older inequality in the Schatten spaces, \eqref{interpolation1} and \eqref{interpolation2}, we obtain
\begin{equation}\label{rho to trace application}
\begin{aligned}
&\int_0^\infty\int_{\mathbb{R}^d}\mathcal{A}_{m,n}(\phi)(t) U(t,x)dxdt\\
&\leq\Big\|\int_0^\infty e^{-it\Delta}|W(t)|e^{it\Delta}dt\Big\|_{\mathfrak{S}^{2d}}^{m-1}\Big\|\int_0^\infty e^{-it\Delta}|W(t)|e^{it\Delta}dt\langle\nabla\rangle^{-\tilde{\beta}}\Big\|_{\mathfrak{S}^{\frac{2d}{d-1}}}\\
&\quad\cdot\|(1-\Delta)^{\tilde{\beta}}\gamma_f\|_{\mathcal{B}(L^2)}\Big\|\langle\nabla\rangle^{-\tilde{\beta}}\int_0^\infty e^{-it\Delta}|W(t)|e^{it\Delta}dt\Big\|_{\mathfrak{S}^{\frac{2d}{d-1}}}\\
&\quad\cdot\Big\|\int_0^\infty e^{-it\Delta}|W(t)|e^{it\Delta}dt\Big\|_{\mathfrak{S}^{2d}}^{n-1}\Big\|\int_0^\infty e^{-it\Delta}|w_2*U(t)|e^{it\Delta}dt\Big\|_{\mathfrak{S}^{2d}}\\
&\leq(c\|W\|_{L_t^2L_x^d})^{m+n-2}(c\|W\|_{L_t^2L_x^2})^2\cdot \|(1+|\cdot|)^{\tilde{\beta}}f\|_{L^\infty}\cdot c\|w_2*U\|_{L_t^2L_x^d}\quad\textup{($W=w_1*\phi$)}\\
&\leq c^{m+n+1}\|\hat{w}_1\|_{L_x^{\frac{2d}{d-2}}}^{m+n-2}\|\hat{w}_1\|_{L^\infty}^2\|\hat{w}_2\|_{L^{\frac{2d}{d-2}}}\|(1+|\cdot|)^\beta f\|_{L^\infty}\|\phi\|_{L_t^2L_x^2}^{m+n}\|U\|_{L_t^2L_x^2},
\end{aligned}
\end{equation}
where $\mathcal{B}(L^2)$ is the operator norm and in the last step, we used that if $r\geq 2$, 
\begin{equation}\label{HY}
\begin{aligned}
\|w*\phi\|_{L^r}&\leq \|\widehat{w*\phi}\|_{L^{r'}}=\|\hat{w}\hat{\phi}\|_{L^{r'}}\quad\textup{(by Hausdorff-Young)}\\
&\leq \|\hat{w}\|_{L^{\frac{2r}{r-2}}}\|\hat{\phi}\|_{L^2}=\|\hat{w}\|_{L^{\frac{2r}{r-2}}}\|\phi\|_{L^2}\quad\textup{(by Plancherel)}.
\end{aligned}
\end{equation}
When either $m=0$ or $n=0$, we give the negative derivative $\langle\nabla\rangle^{-\tilde{\beta}}$ to the integral having $U$ and use \eqref{interpolation2} for that term. Then, estimating as above, we can show that 
\begin{equation}
\begin{aligned}
&\int_0^\infty\int_{\mathbb{R}^d}\mathcal{A}_{m,n}(\phi)(t) U(t,x)dxdt\\
&\leq c^{m+n}\|\hat{w}_1\|_{L_x^{\frac{2d}{d-2}}}^{m+n-1}\|\hat{w}_1\|_{L^\infty}\|\hat{w}_2\|_{L^\infty}\|(1+|\cdot|)^\beta f\|_{L^\infty}\|\phi\|_{L_t^2L_x^2}^{m+n}\|U\|_{L_t^2L_x^2}.
\end{aligned}
\end{equation}
Therefore, by duality, we complete the proof of \eqref{bound on A1} for higher order terms.

It remains to consider the case $m=n=1$. In this case, the inequalities \eqref{interpolation1} and \eqref{interpolation2} does not suffice. Indeed, the first inequality in \eqref{rho to trace application} requires $\frac{1}{2d}\cdot(m+n-1)+\frac{d-1}{2d}\cdot2\geq 1$, i.e., $m+n\geq 3$. Thus, motivated by Corollary \ref{Bound on IV}, in order to upgrade summability, we put negative derivatives on the last term, 
\begin{equation}\label{A11}
\begin{aligned}
&\int_0^\infty\int_{\mathbb{R}^d}\mathcal{A}_{1,1}(\phi)(t) U(t,x)dxdt\\
&=\textup{Tr}\Big(\int_0^\infty\int_0^t \int_0^t e^{-it_1\Delta}W(t_1)e^{it_1\Delta}\gamma_fe^{-it_1'\Delta}W(t_1')e^{it_1'\Delta} e^{-it\Delta}(w_2*U)(t)e^{it\Delta}
dt_1'dt_1dt\Big)\\
&=\textup{Tr}\Big(\int_0^\infty\int_0^t \int_0^t  \langle\nabla\rangle^{\beta_0}e^{-it_1\Delta}W(t_1)e^{it_1\Delta}\gamma_fe^{-it_1'\Delta}W(t_1')e^{it_1'\Delta} \langle\nabla\rangle^{\beta_0}\\
&\quad\quad\quad\quad\quad\quad\quad\cdot \langle\nabla\rangle^{-\beta_0}e^{-it\Delta}(w_2*U)(t)e^{it\Delta} \langle\nabla\rangle^{-\beta_0}dt_1'dt_1dt\Big).
\end{aligned}
\end{equation}
We now claim that 
\begin{equation}\label{A11 claim}
\begin{aligned}
&\textup{Tr}\Big(\int_0^\infty\int_0^t \int_0^t  e^{-it_1\Delta}V_1(t_1)e^{it_1\Delta}\gamma_fe^{-it_1'\Delta}V_2(t_1')e^{it_1'\Delta}\\
&\quad\quad\quad\quad\quad\quad\quad\cdot \langle\nabla\rangle^{-\beta_0}e^{-it\Delta}(w_2*U)(t)e^{it\Delta} \langle\nabla\rangle^{-\beta_0}dt_1'dt_1dt\Big)\\
&\lesssim \|\hat{w}_2\|_{L^{\frac{2d}{d-3}}}\|V_1\|_{L_t^2 L_x^2} \|V_2\|_{L_t^2 L_x^2} \|U\|_{L_t^2 L_x^2}.
\end{aligned}
\end{equation}
Indeed, by complex interpolation between \eqref{interpolation1} and \eqref{eq: Bound on IV} with $\alpha_1=\alpha_2>\frac{d-1}{4}$, we have
\begin{equation}\label{A11 claim'}
\Big\|\langle\nabla\rangle^{-\beta_0}\int_0^\infty e^{-it\Delta}V(t)e^{it\Delta}dt\langle\nabla\rangle^{-\beta_0}\Big\|_{\mathfrak{S}^{d}}\lesssim \|V\|_{L_t^2L_x^\frac{2d}{3}}.
\end{equation}
Thus, repeating \eqref{rho to trace application} but using \eqref{A11 claim'} instead of \eqref{interpolation2},
\begin{equation}
\begin{aligned}
&\textup{Tr}\Big(\int_0^\infty\int_0^t \int_0^t  e^{-it_1\Delta}V_1(t_1)e^{it_1\Delta}\gamma_fe^{-it_1'\Delta}V_2(t_1')e^{it_1'\Delta}\\
&\quad\quad\quad\quad\quad\quad\quad\cdot \langle\nabla\rangle^{-\beta_0}e^{-it\Delta}(w_2*U)(t)e^{it\Delta}\langle\nabla\rangle^{-\beta_0}dt_1'dt_1dt\Big)\\
&\leq\Big\|\int_0^\infty e^{-it\Delta}|V_1(t)|e^{it\Delta}dt\langle\nabla\rangle^{-\tilde{\beta}}\Big\|_{\mathfrak{S}^{\frac{2d}{d-1}}}\|(1-\Delta)^{\tilde{\beta}}\gamma_f\|_{\mathcal{B}(L^2)}\\
&\quad\cdot\Big\|\langle\nabla\rangle^{-\tilde{\beta}}\int_0^\infty e^{-it\Delta}|V_2(t)|e^{it\Delta}dt\Big\|_{\mathfrak{S}^{\frac{2d}{d-1}}}\\
&\quad\cdot\Big\|\langle\nabla\rangle^{-\beta_0}\int_0^\infty e^{-it\Delta}|w_2*U(t)|e^{it\Delta}dt\langle\nabla\rangle^{-\beta_0}\Big\|_{\mathfrak{S}^{d}}\\
&\lesssim \|V_1\|_{L_t^2L_x^2}\|V_2\|_{L_t^2L_x^2}\|(1+|\cdot|)^{\tilde{\beta}} f\|_{L^\infty} \|w_2*U\|_{L_t^2L_x^\frac{2d}{3}}\\
&\lesssim \|\hat{w}_2\|_{L^{\frac{2d}{d-3}}}\|(1+|\cdot|)^{\tilde{\beta}} f\|_{L^\infty}\|V_1\|_{L_t^2 L_x^2} \|V_2\|_{L_t^2 L_x^2} \|U\|_{L_t^2 L_x^2},
\end{aligned}
\end{equation}
where in the last step, we used \eqref{HY}.

Next, distributing derivatives $1-\Delta=1-\sum_{j=1}^d\partial_{x_j}^2$ and then applying \eqref{A11 claim}, we can obtain 
\begin{align*}
&\textup{Tr}\Big(\int_0^\infty\int_0^t \int_0^t  (1-\Delta)e^{-it_1\Delta}V_1(t_1)e^{it_1\Delta}\gamma_fe^{-it_1'\Delta}V_2(t_1')e^{it_1'\Delta}(1-\Delta)\\
&\quad\quad\quad\quad\quad\quad\quad\cdot \langle\nabla\rangle^{-\beta_0}e^{-it\Delta}(w_2*U)(t)e^{it\Delta} \langle\nabla\rangle^{-\beta_0}dt_1'dt_1dt\Big)\\
&\lesssim \|(1-\Delta)V_1\|_{L_t^2 L_x^2} \|(1-\Delta)V_2\|_{L_t^2 L_x^2} \|(1+|\cdot|)^{{\tilde{\beta}}+2} f\|_{L^\infty} \|w_2*U\|_{L_t^2 L_x^{\frac{2d}{3}}}\\
&\leq \|\hat{w}_2\|_{L^{\frac{2d}{d-3}}}\|(1+|\cdot|)^{\beta} f\|_{L^\infty}\|(1-\Delta)V_1\|_{L_t^2 L_x^2} \|(1-\Delta)V_2\|_{L_t^2 L_x^2} \|U\|_{L_t^2 L_x^2}.
\end{align*}
Hence, interpolating it with \eqref{A11 claim}, we get
\begin{equation}
\begin{aligned}
&\textup{Tr}\Big(\int_0^\infty\int_0^t \int_0^t  \langle\nabla\rangle^{\beta_0}e^{-it_1\Delta}V_1(t_1)e^{it_1\Delta}\gamma_fe^{-it_1'\Delta}V_2(t_1')e^{it_1'\Delta} \langle\nabla\rangle^{\beta_0}\\
&\quad\quad\quad\quad\quad\quad\quad\cdot \langle\nabla\rangle^{-\beta_0}e^{-it\Delta}(w_2*U)(t)e^{it\Delta} \langle\nabla\rangle^{-\beta_0}dt_1'dt_1dt\Big)\\
&\lesssim \|V_1\|_{L_t^2 H_x^{\beta_0}} \|V_2\|_{L_t^2 H_x^{\beta_0}} \|w_2*U\|_{L_t^2 L_x^{\frac{2d}{3}}}\\
&\leq \|\hat{w}_2\|_{L^{\frac{2d}{d-3}}}\|(1+|\cdot|)^{\beta} f\|_{L^\infty}\|V_1\|_{L_t^2 H_x^{\beta_0}} \|V_2\|_{L_t^2 H_x^{\beta_0}}\|U\|_{L_t^2 L_x^2}.
\end{aligned}
\end{equation}
Finally, coming back to \eqref{A11}, applying this inequality, we prove that
\begin{equation}
\begin{aligned}
&\int_0^\infty\int_{\mathbb{R}^d}\mathcal{A}_{1,1}(\phi)(t) U(t,x)dxdt\\
&\lesssim\|\hat{w}_2\|_{L^{\frac{2d}{d-3}}}\|(1+|\cdot|)^{\beta} f\|_{L^\infty}\|w_1*\phi\|_{L_t^2 H_x^{\beta_0}} \|w_1*\phi\|_{L_t^2 H_x^{\beta_0}}\|U\|_{L_t^2 L_x^2}\\
&\lesssim \|\langle\cdot\rangle^{\beta_0}\hat{w}_1\|_{L^\infty}^2\|\hat{w}_2\|_{L^{\frac{2d}{d-3}}}\|(1+|\cdot|)^{\beta} f\|_{L^\infty}\|\phi\|_{L_t^2 L_x^2}^2\|U\|_{L_t^2 L_x^2}.
\end{aligned}
\end{equation}

For \eqref{bound on A2}, we decompose $\mathcal{W}_{w_1*\phi}^{(n)}(t)-\mathcal{W}_{w_1*\psi}^{(n)}(t)$ into the sum of $n$ integrals,
\begin{equation}
\begin{aligned}
&(-i)^n\int_0^tdt_n\int_0^{t_n}dt_{n-1}\cdots \int_0^{t_2}dt_1e^{-it_n\Delta}\big(w_1*(\phi-\psi)\big)(t_n) e^{it_n\Delta}\\
&\quad\quad\quad\quad\quad\quad\quad\cdot  e^{-it_{n-1}\Delta}(w_1*\phi)(t_{n-1}) e^{it_{n-1}\Delta}\cdots e^{-it_1\Delta}(w_1*\phi)(t_1)e^{it_1\Delta}\\
&+(-i)^n\int_0^tdt_n\int_0^{t_n}dt_{n-1}\cdots \int_0^{t_2}dt_1e^{-it_n\Delta}(w_1*\psi)(t_n) e^{it_n\Delta}\\
&\quad\quad\quad\quad\quad\quad\quad\cdot e^{-it_{n-1}\Delta}\big(w_1*(\phi-\psi)\big)(t_{n-1}) e^{it_{n-1}\Delta}\cdots e^{-it_1\Delta}(w_1*\phi)(t_1)e^{it_1\Delta}\\
&+\cdots\\
&+(-i)^n\int_0^tdt_n\int_0^{t_n}dt_{n-1}\cdots \int_0^{t_2}dt_1e^{-it_n\Delta}(w_1*\psi)(t_n) e^{it_n\Delta}\\
&\quad\quad\quad\quad\quad\quad\quad\cdot  e^{-it_{n-1}\Delta}(w_1*\psi)(t_{n-1}) e^{it_{n-1}\Delta}\cdots e^{-it_1\Delta}\big(w_1*(\phi-\psi)\big)(t_1)e^{it_1\Delta}.
\end{aligned}
\end{equation}
Using this sum, we decompose the difference
\begin{equation}
\begin{aligned}
&\mathcal{A}_{m,n}(\phi)(t)-\mathcal{A}_{m,n}(\psi)(t)\\
&=w_2*\rho\Big[e^{it\Delta}\big(\mathcal{W}_{w_1*\phi}^{(m)}(t)-\mathcal{W}_{w_1*\psi}^{(m)}(t)\big)\gamma_f\mathcal{W}_{w_1*\phi}^{(n)}(t)^* e^{-it\Delta}\Big]\\
&\quad+w_2*\rho\Big[e^{it\Delta}\mathcal{W}_{w_1*\psi}^{(m)}(t)\gamma_f\big(\mathcal{W}_{w_1*\phi}^{(n)}(t)-\mathcal{W}_{w_1*\psi}^{(n)}(t)\big) e^{-it\Delta}\Big],
\end{aligned}
\end{equation}
into $(m+n)$ terms. For each term, we estimate as in the proof of \eqref{bound on A1}. Collecting all, we obtain \eqref{bound on A2}.
\end{proof}

\section{Bounds on $\mathcal{B}(\phi)$}

We prove the bounds on the operator
\begin{equation*}
    \mathcal{B}(\phi)(t)=w_2*\rho\Big[e^{it\Delta}\mathcal{W}_{w_1*\phi}(t)Q_0\mathcal{W}_{w_1*\phi}(t)^* e^{-it\Delta}\Big]
\end{equation*}
introduced in \eqref{B operator}.
\begin{proposition}[Bounds on $\mathcal{B}(\phi)$]\label{bound on B}
Let $d\geq 3$, $\alpha>\frac{d-2}{2}$ and $\alpha_0$ be given by \eqref{non-relativistic beta} with $\alpha_1=\alpha_2=\alpha$. Suppose that $w=w_1*w_2$, and $|\cdot|^{1/2}\langle\cdot\rangle^{\alpha_0}\hat{w}_1, |\cdot|^{-1/2}\langle\cdot\rangle^{-\alpha_0}\hat{w}_2\in L^\infty$ . Then, there exist small $\epsilon_{\mathcal{B}}>0$ and large $C_\mathcal{B}, C_{\mathcal{B}}'>0$ such that if $\|\phi\|_{L_{t,x}^2}, \|\psi\|_{L_{t,x}^2}\leq\epsilon_\mathcal{B}$, then
\begin{equation}
\begin{aligned}
\|\mathcal{B}(\phi)\|_{L_{t,x}^2}&\leq C_\mathcal{B}\|Q_0\|_{\mathcal{H}^\alpha},\\
\|\mathcal{B}(\phi)-\mathcal{B}(\psi)\|_{L_{t,x}^2}&\leq C_\mathcal{B}'\|Q_0\|_{\mathcal{H}^\alpha}\|\phi-\psi\|_{L_{t,x}^2}.
\end{aligned}
\end{equation}
The constants $\epsilon_\mathcal{B}$, $C_\mathcal{B}$ and $C_{\mathcal{B}}'$ depend only on $d$, $\||\cdot|^{1/2}\langle\cdot\rangle^{\alpha_0}\hat{w}_1\|_{L^\infty}$ and $\||\cdot|^{-1/2}\langle\cdot\rangle^{-\alpha_0}\hat{w}_2\|_{L^\infty}$.
\end{proposition}

\begin{proof} 
For notational convenience, we denote
\begin{equation} \label{sec6:Qphi} 
Q_\phi(t):=e^{it\Delta}\mathcal{W}_{w_1*\phi}(t)Q_0\mathcal{W}_{w_1*\phi}(t)^* e^{-it\Delta}=\mathcal{U}_{w_1*\phi}(t)Q_0 \mathcal{U}_{w_1*\phi}(t)^*.
\end{equation}
Note that by definition, $\mathcal{B}(\phi)=w_2*\rho_{Q_\phi(t)}$.

Recalling \eqref{linear Schrodinger} and \eqref{eq-W-def-1}, 
we see by differentiating \eqref{sec6:Qphi} in $t$ that $Q_\phi$ solves the following equation: 
$$
i\partial_t Q_\phi=[-\Delta+w_1*\phi, Q_\phi]
$$
with initial data $Q_\phi(0)=Q_0$, equivalently, 
\begin{equation}\label{B equation}
Q_\phi(t)=e^{it\Delta}Q_0 e^{-it\Delta}-i\int_0^t e^{i(t-s)\Delta}[w_1*\phi,Q_\phi](s)e^{-i(t-s)\Delta}ds.
\end{equation}
Hence, by the Strichartz estimates in Theorem \ref{Strichartz estimates for operator kernels}, we get
\begin{equation}
\begin{aligned}
\|Q_\phi\|_{\mathcal{S}^\alpha}&\leq c\|Q_0\|_{\mathcal{H}^\alpha}+c\Big\|\la\nabla_x\ra^\alpha\la\nabla_{x'}\ra^\alpha[w_1*\phi,Q_\phi](t,x,x')\Big\|_{L_t^1L_{x,x'}^2}\\
&\leq c\|Q_0\|_{\mathcal{H}^\alpha}+2c\Big\|\la\nabla_x\ra^\alpha\la\nabla_{x'}\ra^\alpha\Big((w_1*\phi)(t,x)Q_\phi(t,x,x')\Big)\Big\|_{L_t^1 L_{x,x'}^2}\\
&\quad +2c\Big\|\la\nabla_x\ra^\alpha\la\nabla_{x'}\ra^\alpha\Big((w_1*\phi)(t,x')Q_\phi(t,x,x')\Big)\Big\|_{L_t^1L_{x,x'}^2},
\end{aligned}
\end{equation}
where the time interval $[0,+\infty)$ is omitted in the norms for notational convenience. Moreover, applying the triangle inequality, the Strichartz estimate in Theorem \ref{Strichartz estimates for density functions in the non-relativistic case} with $\alpha_1=\alpha_2=\alpha$ to the density of \eqref{B equation}, we get
\begin{equation}
\begin{aligned}
\||\nabla|^{1/2}\rho_{Q_\phi}\|_{L_t^2 H_x^{\alpha_0}}&\leq \||\nabla|^{1/2}\rho_{e^{it\Delta}Q_0 e^{-it\Delta}}\|_{L_t^2 H_x^{\alpha_0}}\\
&\quad+\int_\mathbb{R}\||\nabla|^{1/2}\rho_{e^{i(t-s)\Delta}[w_1*\phi,Q_\phi](s)e^{-i(t-s)\Delta}}\|_{L_t^2 H_x^{\alpha_0}}ds\\
&\leq c\|Q_0\|_{\mathcal{H}^\alpha}+c\int_\mathbb{R}\|e^{-is\Delta}[w_1*\phi,Q_\phi](s)e^{is\Delta}\|_{\mathcal{H}_x^{\alpha}}ds\\
&\leq c\|Q_0\|_{\mathcal{H}^\alpha}+2c\Big\|\la\nabla_x\ra^\alpha\la\nabla_{x'}\ra^\alpha\Big((w_1*\phi)(t,x)Q_\phi(t,x,x')\Big)\Big\|_{L_t^1 L_{x,x'}^2}\\
&\quad +2c\Big\|\la\nabla_x\ra^\alpha\la\nabla_{x'}\ra^\alpha\Big((w_1*\phi)(t,x')Q_\phi(t,x,x')\Big)\Big\|_{L_t^1L_{x,x'}^2}.
\end{aligned}
\end{equation}
By the fractional Leibniz rule and Sobolev inequalities with the choices of $\alpha_0$ and $\alpha$ (both are applied only for the $x$-variable),
\begin{equation}
\begin{aligned}
&\big\|\la\nabla_x\ra^\alpha\la\nabla_{x'}\ra^\alpha\big((w_1*\phi)(t,x)Q_\phi(t,x,x')\big)\big\|_{L_t^1 L_{x,x'}^2}\\
&\lesssim \|(w_1*\phi)\|_{L_t^2L_x^d}\|\la\nabla_x\ra^\alpha\la\nabla_{x'}\ra^\alpha Q_\phi(t,x,x')\|_{L_t^2 L_x^{\frac{2d}{d-2}}L_{x'}^2}\\
&\quad+\||\nabla|^\alpha (w_1*\phi)\|_{L_{t,x}^2}\|\la\nabla_{x'}\ra^\alpha Q_\phi(t,x,x')\|_{L_t^2 L_x^\infty L_{x'}^2}\\
&\lesssim\||\nabla|^{1/2}(w_1*\phi)\|_{L_t^2H_x^{\alpha_0}}\|\la\nabla_x\ra^\alpha\la\nabla_{x'}\ra^\alpha Q_\phi(t,x,x')\|_{L_t^2 L_x^{\frac{2d}{d-2}}L_{x'}^2}\\
&\leq\||\cdot|^{1/2}\langle\cdot\rangle^{\alpha_0}\hat{w}_1\|_{L^\infty}\|\phi\|_{L_{t,x}^2}\|Q_\phi\|_{\mathcal{S}^\alpha}.
\end{aligned}
\end{equation}
We estimate $(w_1*\phi)(t,x')Q_\phi(t,x,x')$ in a similar way, interchanging $x$ and $x'$. Thus, we prove that if $\|\phi\|_{L_{t,x}^2}\leq\epsilon_\mathcal{B}$, then
\begin{equation}\label{B equation contraction}
\begin{aligned}
\|Q_\phi\|_{\mathcal{S}^\alpha}+\||\nabla|^{1/2}\rho_{Q_\phi}\|_{L_t^2 H_x^{\alpha_0}}&\leq 2c\|Q_0\|_{\mathcal{H}^\alpha}+2c'\||\cdot|^{1/2}\langle\cdot\rangle^{\alpha_0}\hat{w}_1\|_{L^\infty}\|\phi\|_{L_{t,x}^2}\|Q_\phi\|_{\mathcal{S}^\alpha}\\
&\leq 2c\|Q_0\|_{\mathcal{H}^\alpha}+2c'\epsilon_\mathcal{B}\||\cdot|^{1/2}\langle\cdot\rangle^{\alpha_0}\hat{w}_1\|_{L^\infty}\|Q_\phi\|_{\mathcal{S}^\alpha}.
\end{aligned}
\end{equation}
We take $\epsilon_\mathcal{B}:=\frac{1}{4c'\||\cdot|^{1/2}\langle\cdot\rangle^{\alpha_0}\hat{w}_1\|_{L^\infty}}$. Then, we get
\begin{equation}\label{Q phi bound}
\|Q_\phi\|_{\mathcal{S}^\alpha}+\||\nabla|^{1/2}\rho_{Q_\phi}\|_{L_t^2 H_x^{\alpha_0}}\leq 4c\|Q_0\|_{\mathcal{H}^\alpha}.
\end{equation}
As a result, by \eqref{HY}, we conclude that 
\begin{equation}
\|\mathcal{B}(\phi)\|_{L_{t,x}^2}=\|w_2*\rho_{Q_\phi}\|_{L_{t,x}^2}\leq \||\cdot|^{-1/2}\langle\cdot\rangle^{-\alpha_0}\hat{w}_2\|_{L^\infty}\||\nabla|^{1/2}\rho_{Q_\phi}\|_{L_t^2 H_x^{\alpha_0}}\leq C_\mathcal{B}\|Q_0\|_{\mathcal{H}^\alpha},
\end{equation}
where $C_{\mathcal{B}}=4c\||\cdot|^{-1/2}\langle\cdot\rangle^{-\alpha_0}\hat{w}_2\|_{L^\infty}$.

For the difference
\begin{equation}
\begin{aligned}
Q_\phi(t)-Q_\psi(t)&=-i\int_0^t e^{i(t-s)\Delta}\big[w_1*(\phi-\psi), Q_\phi\big](s)e^{-i(t-s)\Delta}ds\\
&\quad-i\int_0^t e^{i(t-s)\Delta}\big[w_1*\psi, Q_\phi-Q_\psi\big](s)e^{-i(t-s)\Delta}ds,
\end{aligned}
\end{equation}
repeating the estimates in the proof of \eqref{B equation contraction}, we prove that if $\|\phi\|_{L_{t,x}^2}, \|\psi\|_{L_{t,x}^2}\leq\epsilon_\mathcal{B}$, then
\begin{equation}
\begin{aligned}
&\|Q_\phi-Q_\psi\|_{\mathcal{S}^\alpha}+\||\nabla|^{1/2}\rho_{Q_\phi-Q_\psi}\|_{L_t^2 H_x^{\alpha_0}}\\
&\leq c'\||\nabla|^{1/2}w_1*(\phi-\psi)\|_{L_t^2H_x^{\alpha_0}}\|Q_\phi\|_{\mathcal{S}^\alpha}+c'\||\nabla|^{1/2}w_1*\psi\|_{L_t^2H_x^{\alpha_0}}\|Q_\phi-Q_\psi\|_{\mathcal{S}^\alpha}\\
&\leq c'\||\cdot|^{1/2}\langle\cdot\rangle^{\alpha_0}\hat{w}_1\|_{L^\infty}\|\phi-\psi\|_{L_{t,x}^2}\|Q_\phi\|_{\mathcal{S}^\alpha}\\
&\quad+c'\||\cdot|^{1/2}\langle\cdot\rangle^{\alpha_0}\hat{w}_1\|_{L^\infty}\|\psi\|_{L_{t,x}^2}\|Q_\phi-Q_\psi\|_{\mathcal{S}^\alpha}\\
&\leq c'\||\cdot|^{1/2}\langle\cdot\rangle^{\alpha_0}\hat{w}_1\|_{L^\infty}\|\phi-\psi\|_{L_{t,x}^2}\cdot 4c\|Q_0\|_{\mathcal{H}^\alpha}\quad\textup{(by \eqref{Q phi bound})}\\
&\quad+c'\||\cdot|^{1/2}\langle\cdot\rangle^{\alpha_0}\hat{w}_1\|_{L^\infty}\cdot\epsilon_\mathcal{B}\cdot\|Q_\phi-Q_\psi\|_{\mathcal{S}^\alpha}.
\end{aligned}
\end{equation}
By the choice of $\epsilon_\mathcal{B}$, 
\begin{equation}
\begin{aligned}
\||\nabla|^{1/2}\rho_{Q_\phi-Q_\psi}\|_{L_t^2 H_x^{\alpha_0}}\leq 4cc'\||\cdot|^{1/2}\langle\cdot\rangle^{\alpha_0}\hat{w}_1\|_{L^\infty}\|Q_0\|_{\mathcal{H}^\alpha}\|\phi-\psi\|_{L_{t,x}^2}.
\end{aligned}
\end{equation}
Thus, by \eqref{HY}, we conclude that 
\begin{equation}
\begin{aligned}
\|\mathcal{B}(\phi)-\mathcal{B}(\psi)\|_{L_{t,x}^2}&=\|w_2*(\rho_{Q_\phi}-\rho_{Q_\psi})\|_{L_{t,x}^2}\\
&\leq \||\cdot|^{-1/2}\langle\cdot\rangle^{-\alpha_0}\hat{w}_2\|_{L^\infty}\||\nabla|^{1/2}\rho_{Q_\phi-Q_\psi}\|_{L_t^2 H_x^{\alpha_0}}\\
&\leq C_\mathcal{B}'\|Q_0\|_{\mathcal{H}^\alpha}\|\phi-\psi\|_{L_{t,x}^2},
\end{aligned}
\end{equation}
where $C_{\mathcal{B}}'=4cc'\||\cdot|^{1/2}\langle\cdot\rangle^{\alpha_0}\hat{w}_1\|_{L^\infty}\||\cdot|^{-1/2}\langle\cdot\rangle^{-\alpha_0}\hat{w}_2\|_{L^\infty}$.
\end{proof}

\section{Proof of the main theorem}\label{proof of the main theorem}

First, we prove that 
\begin{equation}
\Gamma(\phi)=(1+\mathcal{L})^{-1}\Big\{\sum_{m,n=1}^\infty\mathcal{A}_{m,n}(\phi)+\mathcal{B}(\phi)\Big\}
\end{equation}
is contractive in a small ball in $L_{t,x}^2$. Let $\epsilon>0$ be a sufficiently small number. Suppose that $\|Q_0\|_{\mathcal{H}^\alpha}\leq\epsilon$ and
\begin{equation}
\|\phi\|_{L_{t,x}^2}, \|\psi\|_{L_{t,x}^2}\leq 2 C_{\mathcal{B}}\|1+\mathcal{L}\|_{L_{t,x}^2\to L_{t,x}^2}^{-1}\|Q_0\|_{\mathcal{H}^\alpha}=:R.
\end{equation}
Note that $R$ is also a sufficiently small number, since $\|Q_0\|_{\mathcal{H}^\alpha}$ is assumed to be small. Then, by Proposition \ref{bound on A} and \ref{bound on B}, 
\begin{equation}
\|\Gamma(\phi)\|_{L_{t,x}^2}\leq \|1+\mathcal{L}\|_{L_{t,x}^2\to L_{t,x}^2}^{-1}\Big\{\sum_{m,n=1}^\infty C_{\mathcal{A}}^{m+n+1}\|\la\cdot\ra^\alpha f\|_{L^\infty}R^{m+n}+C_{\mathcal{B}}\|Q_0\|_{\mathcal{H}^\alpha}\Big\}\leq R.
\end{equation}
where in the second inequality, we used that the sum $\sum_{m,n=1}^\infty C_{\mathcal{A}}^{m+n+1}\|\la\cdot\ra^\alpha f\|_{L^\infty}R^{m+n}$ is $O(R^2)$, so it is bounded by $C_{\mathcal{B}}\|Q_0\|_{\mathcal{H}^\alpha}=O(R)$. Similarly, we prove that 
\begin{equation}
\begin{aligned}
&\|\Gamma(\phi)-\Gamma(\psi)\|_{L_{t,x}^2}\\
&\leq \|1+\mathcal{L}\|_{L_{t,x}^2\to L_{t,x}^2}^{-1}\Big\{\sum_{m,n=1}^\infty (m+n)C_{\mathcal{A}}^{m+n+1}\|\la\cdot\ra^\alpha f\|_{L^\infty}(2R)^{m+n-1}+C_{\mathcal{B}}'\epsilon\Big\}\|\phi-\psi\|_{L_{t,x}^2}\\
&\leq\frac{1}{2}\|\phi-\psi\|_{L_{t,x}^2}.
\end{aligned}
\end{equation}
Thus, by the contraction mapping theorem, there exists a unique $\phi\in L_{t,x}^2$ such that $\phi=\Gamma(\phi)$.

Next, we derive the equation \eqref{wave operator formulation} from $\phi=\Gamma(\phi)$. Precisely, we claim that $Q(t)$, defined by
\begin{equation}\label{definition of Q}
Q(t):=e^{it\Delta}\mathcal{W}_{w_1*\phi}(t)(\gamma_f+Q_0)\mathcal{W}_{w_1*\phi}(t)^*e^{-it\Delta}-\gamma_f,
\end{equation}
is a solution to \eqref{wave operator formulation}. Indeed, it follows from the series expansion for the wave operator (see \eqref{wave operator sum}) and its boundedness (see \eqref{boundedness of wave operator}) that $Q(t)$ is well-defined in $\mathfrak{S}^{2d}$. Moreover, we have
\begin{equation}
\begin{aligned}
\|w_2*\rho_Q-\phi\|_{L_{t,x}^2}&=\Big\|-\mathcal{L}(\phi)+\sum_{m,n=1}^\infty\mathcal{A}_{m,n}(\phi)+\mathcal{B}(\phi)-\phi\Big\|_{L_{t,x}^2}\\
&=\Big\|-\mathcal{L}(\phi)+(1+\mathcal{L})(1+\mathcal{L})^{-1}\Big\{\sum_{m,n=1}^\infty\mathcal{A}_{m,n}(\phi)+\mathcal{B}(\phi)\Big\}-\phi\Big\|_{L_{t,x}^2}\\
&=\|-\mathcal{L}(\phi)+(1+\mathcal{L})\Gamma(\phi)-\phi\|_{L_{t,x}^2}\\
&=\|-\mathcal{L}(\phi)+(1+\mathcal{L})\phi-\phi\|_{L_{t,x}^2}=0\quad\textup{(by $\Gamma(\phi)=\phi$)},
\end{aligned}
\end{equation}
where the first identity follows from straightforward calculations using the infinite series expansion of the wave operator and the definitions of $\mathcal{L}$, $\mathcal{A}_{m,n}$ and $\mathcal{B}$. Now, inserting $\phi=w_2*\rho_Q$ into \eqref{definition of Q}, we conclude that $Q$ satisfies the equation \eqref{wave operator formulation},
\begin{equation}
\begin{aligned}
Q(t)&=\mathcal{U}_{w_1*w_2*\rho_Q}(t)(\gamma_f+Q_0)\mathcal{U}_{w_1*w_2*\rho_Q}(t)^*-\gamma_f\\
&=\mathcal{U}_{w*\rho_Q}(t)(\gamma_f+Q_0)\mathcal{U}_{w*\rho_Q}(t)^*-\gamma_f
\end{aligned}
\end{equation}
in $C_t([0,+\infty);\mathfrak{S}^{2d})$.

Finally, by \eqref{HY}, we prove the desired global-in-time bound, 
\begin{equation}
\begin{aligned}
\|w*\rho_Q\|_{L_t^2 L_x^d}&\leq \|w_1*w_2*\rho_Q\|_{L_t^2 L_x^d}\leq \|\hat{w}_1\|_{L^{\frac{2d}{d-2}}}\|w_2*\rho_Q\|_{L_{t,x}^2}\\
&\leq\|\hat{w}_1\|_{L^{\frac{2d}{d-2}}}\|\hat{w}_2\|_{L^\infty}\|\phi\|_{L_{t,x}^2}<\infty,
\end{aligned}
\end{equation}
which implies scattering in $\mathfrak{S}^{2d}$ by Lemma \ref{key lemma}.

\end{document}